\documentclass[12pt]{article}
\usepackage{enumerate,amsmath,amssymb,amsthm,amscd,eucal}
\usepackage{latexsym}
\usepackage{MnSymbol}
\usepackage{graphicx}
\usepackage[usenames]{color}
\usepackage{caption}
\usepackage{subcaption}
\usepackage{xifthen}

\theoremstyle{plain}

\newtheorem{thm}[equation]{Theorem}
\newtheorem{alg}[equation]{Algorithm}

\newtheorem*{thm*}{Theorem}
\newtheorem{lemma}[equation]{Lemma}
\newtheorem{prop}[equation]{Proposition}
\newtheorem{cor}[equation]{Corollary}
\newtheorem{remark}[equation]{Remark}
\newtheorem{definition}[equation]{Definition}
\newtheorem{example}[equation]{Example}

\numberwithin{equation}{section}

\def\fsl{\mathfrak{sl}}
\def\End{{\rm End}}

\def\ot{\otimes}

\newcommand{\CC}{\mathbb{C}}

\newcommand{\LL}{\mathsf{L}}

\newcommand{\M}{\mathsf{M}}

\newcommand{\Q}{\mathsf{Q}}
\newcommand{\R}{\mathsf{R}}
\newcommand{\smf}{\mathsf{s}}
\newcommand{\SF}{\mathsf{S}}

\newcommand{\T}{\mathsf{T}}

\newcommand{\VV}{\mathsf{V}}
\newcommand{\W}{\mathsf{W}}

\newcommand{\As}{\mathsf{Z}}
\newcommand{\fs}{\mathsf{f}}
\newcommand{\gs}{\mathsf{g}}
\newcommand{\Atwo}{\mathsf{A_2}}
\newcommand{\dimm}{\mathsf{dim}}
\newcommand{\Ann}{\mathsf{A_{r-1}}}
\newcommand{\ve}{\varepsilon}
\newcommand{\EM}{\sf{m}}

\newcommand{\grim}[3][]{%
\ifthenelse{\isempty{#1}}{
\vcenter{\hbox{\scalebox{#2}{\includegraphics{grimfiles/#3}}}}}
{\vcenter{\hbox{\scalebox{#2}[#1]{\includegraphics{grimfiles/#3}}}}}
}

\title{The Combinatorics of $\mathsf{A_2}$-webs}
\author{Georgia Benkart \\
\small {Department of Mathematics}\\
\small{University of Wisconsin-Madison}\\
\small{Madison, WI 53706, USA}\\
\small\tt{benkart@math.wisc.edu}\\
\and
Soojin Cho\thanks{This research was supported by Basic Science Research Program through the
National Research Foundation of Korea(NRF) funded by the Ministry of
Education(NRF2011-0012398).} \\
\small{Department of Mathematics} \\
\small{Ajou University} \\
\small{Suwon,  443-749, Korea (ROK)}\\
\small\tt{chosj@ajou.ac.kr}\\
\and
Dongho Moon\thanks{This research was supported by the Basic Science Research Program of the National Research Foundation of Korea (NRF) funded by the Ministry of Education, Science and Technology (2010-0022003).
The hospitality of the Mathematics Department at the University of Wisconsin-Madison while this research
was done  is gratefully
acknowledged.} \\
\small{Department of Mathematics}\\
\small{Sejong University}\\
\small{Seoul, 133-747, Korea (ROK)}\\
\small\tt{dhmoon@sejong.ac.kr}}
\begin{document}
\maketitle
\begin{abstract}
{The nonelliptic $\mathsf{A_2}$-webs with $k$ ``$+$''s
on the top boundary and $3n-2k$ ``$-$''s on the bottom
boundary combinatorially model the
space $\mathsf{Hom}_{\mathfrak{sl}_3}(\mathsf{V}^{\otimes (3n-2k)}, \mathsf{V}^{\otimes k})$
 of  $\mathfrak{sl}_3$-module maps on tensor powers of the natural $3$-dimensional $\mathfrak{sl}_3$-module
 $\mathsf{V}$,  and they have connections with the combinatorics of
Springer varieties.
Petersen, Pylyavskyy, and Rhodes showed  that the set of such  $\mathsf{A_2}$-webs and
the set of  semistandard tableaux of shape $(3^n)$
and type $\{1^2,\dots,k^2,k+1,\dots, 3n-k\}$ have the same cardinalities.
In this work, we  use the $\sf{m}$-diagrams introduced by Tymoczko and the Robinson-Schensted correspondence
to construct an explicit bijection, different from the one given by Russell,  between these two sets.  In establishing our result, we show that the pair of standard  tableaux constructed using the notion of path depth is
the same as the pair constructed  from  applying the Robinson-Schensted correspondence to a $3\,2\,1$-avoiding
permutation.  We also obtain a bijection between such  pairs of standard tableaux and Westbury's $\Atwo$ flow diagrams.}\end{abstract}

\section{Introduction} \label{sec:intro}

Combinatorial spiders can be used to describe the space of invariants of  tensor products of irreducible representations for Lie
algebras and their quantum analogues.  Kuperberg \cite{K}  gave a description of the spiders for
types   $\mathsf{A_1,A_2, B_2}$,  $\mathsf{G_2}$ using certain combinatorial graphs called {\it webs}.
In this paper, we will focus on webs of type $\mathsf{A_2}$.

An $\mathsf{A_2}$-web is a directed graph such that
(1) vertices on a boundary line have degree one;  (2) internal vertices are trivalent;  and  (3) each vertex is either
a sink (all its edges are directed inward) or a source (all its edges are directed outward).
A boundary vertex that is a sink
corresponds to the natural three-dimensional module $\mathsf{V}$ for $\mathfrak{sl}_3$ or  $\mathsf{U}_q(\mathfrak{sl}_3)$  and carries a  label  ``$+$'';
while a boundary vertex that is a source corresponds to the dual module $\mathsf{V}^*$ and is labeled ``$-$''.
An $\mathsf{A_2}$-web is said to be {\it nonelliptic} if each interior face is bounded by at least six edges.

 Let $\lambda = (\lambda_1 \geq \lambda_2 \geq \cdots \geq \lambda_\ell \geq 0) \vdash d$ be a partition of  the integer $d$.     The corresponding Young diagram is a left-justified array with $\lambda_1$ boxes in the first row, $\lambda_2$ boxes in the second row, and so forth.
A {\it standard tableau} of shape $\lambda$ is a filling of the Young diagram with $d$ distinct positive integers in such a way that the entries in the boxes increase left to right across each row and from top to bottom down each column.   Let   $n_1 < n_2 < \cdots < n_r$  be positive integers.  A  {\it semistandard tableau} of shape $\lambda$ and {\it type}  $\{{n_1}^{a_1}, {n_2}^{a_2}, \dots, {n_r}^{a_r}\}$ such that $\sum_{i=1}^r  a_i = d$  is a filling of
the boxes  of $\lambda$ with $a_1$ entries $n_1$, $a_2$ entries $n_2$ etc.,  so that the entries in the boxes weakly increase left to right
across the rows and strictly increase from top to bottom down each column.
Petersen-Pylyavskyy-Rhoades \cite{PPR} re-interpreted  Kuperberg's result \cite[Thm.~6.1]{K}  on the dimension of the space of invariants (i.e.~the number of  $\mathsf{A_2}$-webs)
in terms of semistandard tableaux of a particular shape and type.

\begin{thm}[\cite{PPR}] \label{thm:prp}
Let $\gamma$ be a fixed boundary with $k$ ``$+$"s and  $3n-2k$ ``$-$"s. The number of semistandard tableaux of shape $(3,3,\ldots)=(3^n)$ and type  $\{1^2,\ldots, k^2, k+1, \ldots, 3n-k\}$ is equal to the number of nonelliptic  $\mathsf{A_2}$-webs with boundary $\gamma$.
\end{thm}

In \cite{T}, Tymoczko gave a bijective proof of Theorem~\ref{thm:prp} for the case  $k=0$ by exhibiting an explicit bijection between the set of standard tableaux of shape $(3^n)$ and the set of nonelliptic  $\mathsf{A_2}$-webs which have  $3n$  ``$-$"s on the boundary.  Her approach introduces the notion of
an {\it $\EM$-diagram}.   An $\EM$-diagram on $n$ vertices consists of a boundary line labeled with the numbers $1,2,\ldots,n$, together with a collection of arcs drawn above it.   For a given standard tableau $\T$ of shape $(3^n)$, Tymoczko associates an $\EM$-diagram to $\T$ by the following algorithm:
\begin{alg}\label{alg:tym} \smallskip

\begin{enumerate}
\item[{\rm (1)}]  Take the conjugate of $\T$ (transpose the rows and columns) to obtain a standard tableau $\T^{\tt t}$ of shape $(n,n,n)$.
\item[{\rm (2)}]  Starting with the bottom row of $\T^{\tt t}$,  read  the entries of $\T^{\tt t}$ from left to right
across the  rows.  If $\ell$ is an entry in this reading,  connect  vertex $\ell$ to vertex $j$ on the boundary line of the $\EM$-diagram  with an arc,  where $j$ is the largest number less than $\ell$ in the row immediately  above the row with $\ell$ that has not
 already been connected with an arc.  After finishing with the bottom row,  move to the entries in the second row of $\T^{\tt t}$ and repeat the same procedure.
 \end{enumerate} \end{alg}

\begin{example}\label{exam:first}
If  $\T =\grim{0.8}{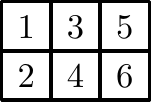}\,$, then  $\T^{\tt t} =\grim{0.8}{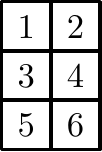}\,$.
Starting with  $5$ in the bottom row, connect it to $4$,  and then connect  $6$ to $3$.  Moving to the second row,
begin with $3$ and connect it to $2$, since $2$ is the largest entry less than $3$ in the first row.
Finally,  connect $4$ and  $1$. The resulting $\EM$-diagram associated to $\T$ is  $$
\grim{0.7}{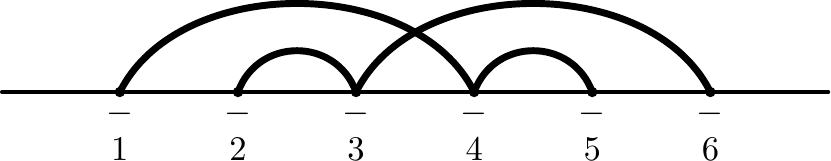}.
$$

\end{example}

The next step of Tymoczko's procedure is to
replace vertices in the $\EM$-diagram with certain configurations as pictured below.
$$
 \grim{0.55}{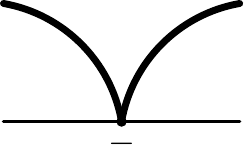} \longrightarrow \grim{0.55}{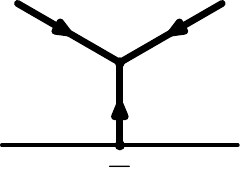}, \ \quad \grim{0.55}{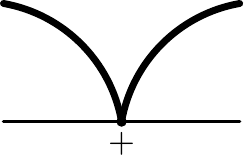} \longrightarrow \grim{0.55}{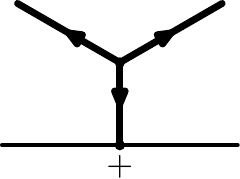},
\ \quad   \grim{0.55}{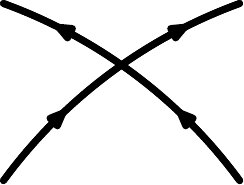}  \longrightarrow
 \grim{0.55}{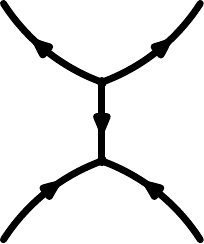}.
$$
This is called \emph{trivalizing} in \cite{T}, because vertices off the boundary  become \emph{trivalent} by this process.
For Example~\ref{exam:first} above,  the corresponding $\mathsf{A_2}$-web is
$$
\grim{0.6}{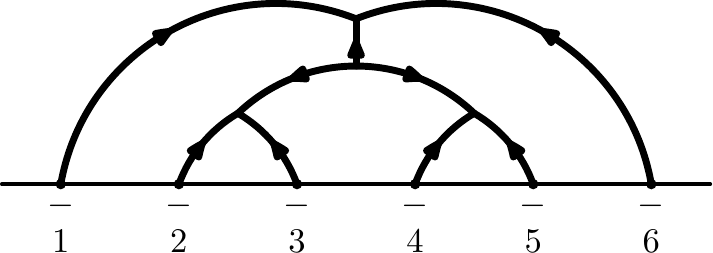}.
$$

In \cite{Ru}, Russell exhibited a bijection between the set of semistandard tableaux  of shape
$(3,3,\ldots)=(3^n)$ and type  $\{1^2,\ldots, k^2, k+1, \ldots, 3n-k\}$ and the nonelliptic
$\mathsf{A_2}$-webs  with one boundary line having $k$ ``$-$'' vertices  and $ 3n-2k$  ``$+$''
vertices for arbitrary $k$ such that $2k \leq  3n$.   In this paper, we demonstrate an explicit
bijection between the set of such semistandard tableaux and the nonelliptic $\Atwo$-webs having
$k$  ``$+$'' vertices on the top boundary line and $3n-2k$ ``$-$'' vertices on the bottom boundary
line. Such webs combinatorially model the space
$\mathsf{Hom}_{\mathfrak{sl}_3}(\VV^{\ot (3n-2k)},  \VV^{\ot k})$.
Several noteworthy special cases are when  $k=0$, which corresponds to the
space $\mathsf{Hom}_{\mathfrak{sl}_3}(\VV^{\ot 3n},  \mathbb C) = \mathsf{Inv}(\VV^{\ot 3n})$  of
$\mathfrak{sl}_3$-invariants in $\VV^{\ot 3n}$, and when  $n=k$, which corresponds to the centralizer algebra
$\End_{\mathfrak{sl}_3}(\VV^{\ot k})$ of the $\mathfrak{sl}_3$-action on the tensor space $\VV^{\ot k}$.

Our approach uses classical results from the representation theory of the symmetric group $\SF_\ell$ such as the well-known Robinson-Schensted correspondence (RS correspondence).
That the  representation theory of   $\SF_\ell$ plays an essential role is no surprise,  as  $\SF_\ell$ is closely related with the representation theory and invariant theory of the general linear and special linear groups via Schur-Weyl duality.
Another example of  the connection with the RS correspondence  can be found  in \cite{HRT}.

The RS correspondence (see for example, \cite[Sec.~3.1]{S} or \cite{W})  is a bijection between  permutations and ordered pairs of standard tableaux of the same shape,  consisting of an insertion tableau and a recording tableau.   Here,  we reverse the order of the standard tableaux in  the pair, so that  a permutation  $\sigma$ corresponds to $(\R,\Q)$,  where $\R$ is the recording tableau and $\Q$ is the insertion tableau.
This switch is simply a notational convenience for when we show in what follows that  the pair of standard tableaux constructed using the notion of path depth is the same as the pair $(\R,\Q)$.

For example, suppose $\sigma=\begin{pmatrix} 1 &2 &3&4&5 \\ 4&2&1&5&3 \end{pmatrix}$.
The tableau $\Q$ is created by inserting the numbers $ 4\,2\,1\,5\,3$ in succession,  and $\R$ is created by recording where the new box  has been added to $\Q$ at each step.  In the insertion process, a number $x$ is adjoined  to the right-hand end of first row of $\Q$ unless that violates the ``standard'' property. When a violation occurs,   the smallest $y$ in the row such that $x < y$ is bumped to the next row,  and $x$ is placed in the position vacated by $y$.  The insertion-bumping process then continues with  $y$ until a number is adjoined to the end of a row in the tableau.    For $\sigma$,  the steps to construct $(\R,\Q)$ are as follows:
$$
\left(\grim{0.8}{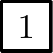} \, ,\, \grim{0.8}{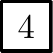} \right) \longrightarrow
\left(\grim{0.8}{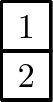}\, , \, \grim{0.8}{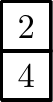} \right) \longrightarrow
\left(\grim{0.8}{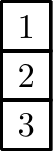}\, , \, \grim{0.8}{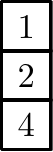} \right) \longrightarrow
\left(\grim{0.8}{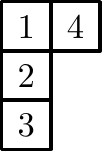}\, , \, \grim{0.8}{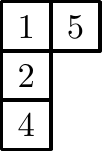} \right) \longrightarrow
\left(\grim{0.8}{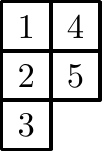}\, , \, \grim{0.8}{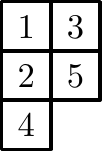} \right)
$$

In Section~\ref{sec:central}, we review basic definitions  related to  $\Atwo$-webs and their connections with centralizer algebras of the $\mathfrak{sl}_3$-action on tensor products.  In Section~\ref{sec:bijection}, we construct a map $\Phi$ from semistandard tableaux to $\Atwo$-webs,  and then prove in Section~\ref{sec:proof} that the webs are nonelliptic. The results on  $3\,2\,1$-avoiding permutations in Section~\ref{sec:combinatorics} are used in Section \ref{sec:inverse}  to show  that $\Phi$  is indeed a bijection.   In Section~\ref{sec:flow}, we illustrate how  Westbury's flow diagrams  provide  a different bijection, which
avoids using  the RS correspondence.   We obtained this alternate bijection after submitting the first version of this
paper, and thank Heather Russell for bringing \cite{We} to our attention.

We note that the result on shuffle and join obtained in  \cite[Thm.~3]{Ru}   as an application of Russell's bijection  can be proved using our bijection also, although we do not include that proof here.

\section{$\Atwo$-webs and centralizer algebras} \label{sec:central}

The $\Atwo$-spider category  introduced by Kuperberg in \cite{K}  describes the invariants of tensor products of copies of  the natural 3-dimensional representation $\VV_+ = \VV$ of the complex Lie algebra $\mathfrak{sl}_3$ and its dual representation $\VV_- = \VV^*$.   A tensor product of  copies of $\VV_+$ and $\VV_-$  corresponds to a string
$\smf = \smf_1\cdots \smf_n$   of symbols $+$ and $-$, and we write $\VV_{\smf} = \VV_{\smf_1}\ot \cdots \ot\VV_{\smf_n}$.
 Let $\mathsf{Inv}(\VV_{\smf}) = \{ w \in \VV_{\smf} \mid  x.w = 0$ for all $x
\in \mathfrak{sl}_3\}$ denote the space of $\mathfrak{sl}_3$-invariants in $\VV_{\smf}$.

The objects in the $\Atwo$-spider category are finite strings of $+$ and $-$ signs, including
the empty string. The morphisms of the $\Atwo$-spider category are $\Atwo$-webs.
An $\Atwo$-web is defined  combinatorially as
a directed graph with a boundary labeled by the sign string  $\smf = \smf_1\cdots \smf_n$
so that
\begin{enumerate}
\item vertices on the boundary have degree one and are labeled by the components of
$\smf$ in a counterclockwise fashion;
\item internal vertices are trivalent;  and
\item  each vertex is either
a sink (all its edges are directed inward) or a source (all its edges are directed outward).
\end{enumerate}    Such an $\mathsf{A_2}$-web is said to be {\it nonelliptic} if each interior face is bounded by at least six edges. The following is an example of nonelliptic $\Atwo$-web with  boundary $\smf =---+-+-$:
$$
\grim{0.7}{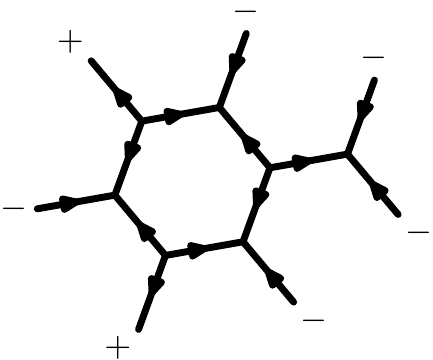}
$$
No distinction is made between two diagrams that are the same up to rotation.

Let  $\mathsf{B}(\smf)$ denote the set of nonelliptic $\Atwo$-webs with  boundary condition $\smf$. The vector space of formal linear combinations of $\mathsf{B}(\smf)$  is the $\Atwo$-web space $\W(\smf)$.

\begin{thm}\label{T:inv}{\rm (\cite[Thm.~6.1]{K})}
Assume $\smf$ is a sign string.  Then the vector spaces $\W(\smf)$ and $\mathsf{Inv}(\VV_{\smf})$ have
 the same dimensions.
\end{thm}

Consider the special case  that  $\smf$ consists of $k$ minuses followed by $k$ pluses, so that
$\VV_{\smf} = (\VV^*)^{\ot k} \ot  \VV^{\ot k} \cong  \big (\VV^{\ot k} \big)^{\ast} \ot  \VV^{\ot k}
\cong \mathsf{End}(\VV^{\ot k})$.   In this situation,    $\mathsf{Inv}(\VV_{\smf}) \cong \mathsf{Inv}\big(\mathsf{End}(\VV^{\ot k})\big) = \mathsf{End}_{\mathfrak{sl}_3}(\VV^{\ot k}) = \{ \phi \in \End(\VV^{\ot k})
\mid   x \phi = \phi x \ \hbox{\rm for all} \ x \in \mathfrak{sl}_3\}$,  since  the $\mathfrak{sl}_3$-action on  $\End(\VV^{\ot k})$
is given by $(x.\phi)(w) = x.\phi(w) - \phi(x.w)$.

The above theorem implies that  we may label a basis of  the \emph{centralizer algebra}  $\As_k = \mathsf{End}_{\mathfrak{sl}_3}(\VV^{\ot k})$ by nonelliptic
$\Atwo$-webs that have  a bottom row of $k$ minus signs and a top row of $k$ plus signs.
For example, $\; \grim{0.8}{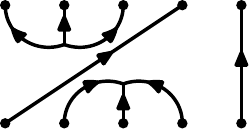}\;$ is an element in $\As_5$.

Multiplication of $\Atwo$-webs is governed by the following rules:
\begin{center}
{(M1)} \; $\grim{0.7}{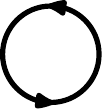}= 3$ \hfill 
{(M2)} \; $\grim{0.7}{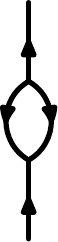}=2 \; \grim{0.7}{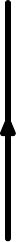}$  \hfill
{(M3)} \; $\grim{0.7}{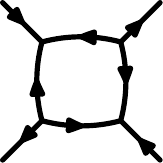}=\grim{0.7}{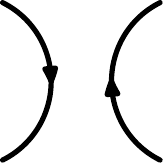}+
\grim{0.7}{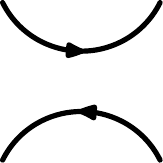}$.
\end{center}
For example,
$$
\left(\;
\grim[0.8]{0.7}{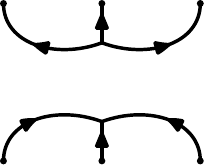}
\; \right)^2 \; = \;
\grim[0.8]{0.7}{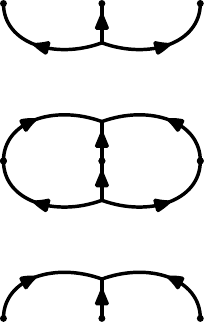}
\; \overset{\text{(M2)}}{=}\;
2\;
\grim[0.8]{0.7}{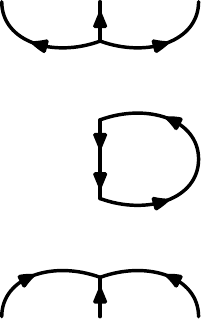}
\; \overset{\text{(M1)}}{=}\; 6\;
\grim[0.8]{0.7}{g1.pdf}
\; .
$$

\medskip

The product of $\; \grim[0.8]{0.7}{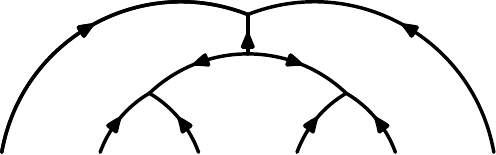}\;$ and
$\; \grim[0.8]{0.7}{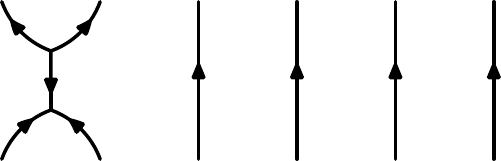}\;$  is computed by placing the first
graph above the second and concatenating the diagrams using the rules above:
$$
\grim[0.8]{0.7}{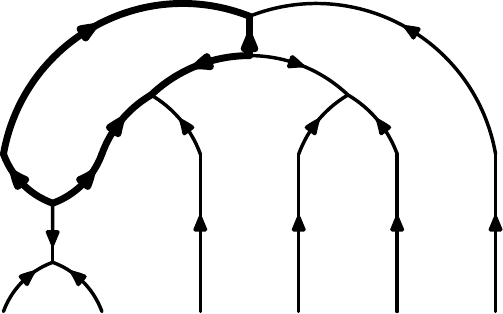}\;   \; \overset{\text{(M3)}}{=} \;
\grim[0.8]{0.7}{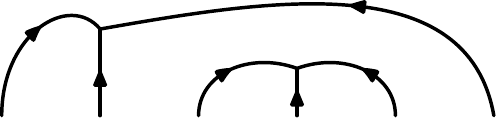}
\; \; + \; \;
\grim[0.8]{0.7}{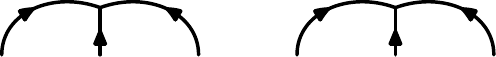} \; .
$$
The thick edge indicates the portion of the web to which the rule is applied.

 The symmetric group $\mathsf{S}_k$ acts on $\VV^{\ot k}$ by permuting the tensor factors, and this action
 commutes with the action of $\mathfrak{sl}_3$.    Classical Schur-Weyl duality says  that there is
 an epimorphism $\CC \mathsf{S}_k \rightarrow \mathsf{End}_{\mathfrak{sl}_3}(\VV^{\ot k}) = \As_k$
 and that the kernel of this map is the ideal generated by the Young symmetrizer
 $\sum_{\sigma} (-1)^{|\sigma |} \sigma$, where  the sum is over  the permutations $\sigma$ of
 $\{1,2,3,4\}$.

The transpositions   $s_i = (i \:i+1)$  for $i=1,\dots, k-1$ generate $\mathsf{S}_k$.    We let  $\fs_i$ for $1 \leq i < k$,  denote the image of $\mathsf{id}-s_i$ in $\As_k$.
 Then,  the elements $\fs_i, \ 1 \leq i < k$, generate $\As_k$,  and they satisfy the following relations:
\begin{itemize}
\item[{(R1)}]   $\fs_i^2  = 2  \fs_i$;
\item[{(R2)}]   $\fs_i \fs_j  = \fs_j \fs_i$  \qquad  if  \ \ $| i - j| > 1$;
\item[{(R3)}]   $\fs_i \fs_{i+1} \fs_i - \fs_i  =  \fs_{i+1} \fs_i \fs_{i+1} - \fs_{i+1}$  \quad for all \  $i=1,\dots, k-2$;
\item[{(R4)}]  $\gs_i \gs_{i+1} \gs_i - 4 \gs_i = 0$, \
where \ $\gs_i = \fs_i \fs_{i+1} \fs_i - \fs_i  =   \fs_{i+1} \fs_i \fs_{i+1} - \fs_{i+1}$ \quad for all \  $i=1,\dots, k-3$.
\end{itemize}
Relation  (R4) comes from the fact that if $f_i = \mathsf{id}-s_i$ and $g_i = f_i f_{i+1}f_i - f_i$  in $\CC \SF_k$,  then    $g_i g_{i+1}g_i - 4g_i = 8 \sum_{\tau} (-1)^{|\tau |} \tau
= 8\psi\left(\sum_{\sigma} (-1)^{|\sigma |} \sigma\right)\psi^{-1}$,  as $\tau$ ranges over the permutations
of ${i,i+1,i+2, i+3}$,  and $\sigma$ ranges over the permutations
of $\{1,2,3,4\}$.   Here $\psi$ is the element of $\SF_k$ that interchanges $\ell$ and $\ell+i-1$ for
$\ell=1,2,3,4$ and acts as the identity on the other  values.
The following relations can be deduced from (R1)-(R4):
\begin{itemize}
\item[{(Ra)}]   $\fs_i \gs_j  = \gs_j \fs_i $  \qquad  for all \ $i=1,\dots,k-1$,  \ $j = 1, \dots, k-2$;
\item[{(Rb)}]     $\gs_j^2  = 6  \gs_j$;
\item[{(Rc)}]     $\gs_j \gs_{j+2} \gs_j  = -2 \gs_j,   \qquad   \gs_{j+2} \gs_j \gs_{j+2} = -2  \gs_{j+2}$;
\item[{(Rd)}]   $\fs_i \gs_i  =  -2 \gs_i   =  \fs_{i+1}\gs_i$.
\end{itemize}

Each of the generators $\fs_i$  can be identified with  a nonelliptic $\Atwo$-web as shown below,
and one can verify that the multiplication above
corresponds to web multiplication.
\begin{align*}  \fs_i  \ \ \leftrightarrow&\ \ \grim{0.8}{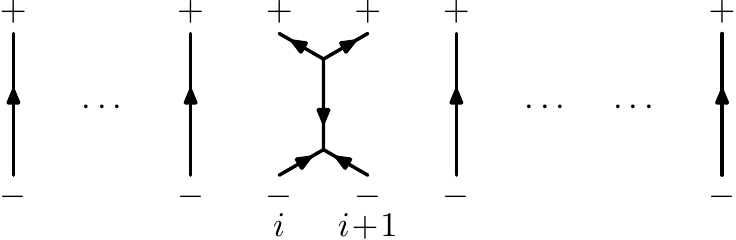}      \\
\gs_i \ \  \leftrightarrow&\ \ \grim{0.8}{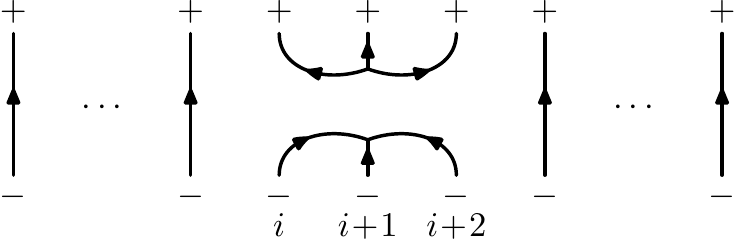}. \end{align*}
The following figure illustrates $\fs_1 \fs_{2} \fs_1 = \gs_1 +\fs_1$ (equivalently, $\gs_1 = \fs_1 \fs_2\fs_1 - \fs_1$):
$$
\grim{0.7}{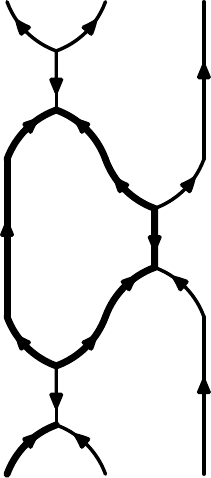} \;  \overset{\text{(M3)}}{=}  \; \grim{0.7}{g1.pdf} \; +\; \grim{0.7}{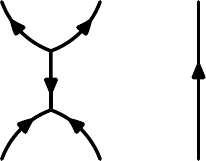}  \;.
$$

\section{A map between semistandard tableaux and \break nonelliptic $\Atwo$-webs} \label{sec:bijection}

\subsection{A bijection between semistandard tableaux and standard tableaux}
Let   $\nu =(\nu_1\geq \nu_2 \geq  \ldots \geq \nu_k) \subseteq (n,n,\ldots)=(n^k)$ be a partition which
sits inside the partition $(n^k)$  having $k$ rows of $n$ boxes each.
The  \emph{complementary partition} $\nu^c = (n-\nu_k \geq  n-\nu_{k-1}\geq \dots \geq  n-\nu_1)$ also fits inside $(n^k)$ in the lower right-hand corner by rotating $\nu^c$  by $180$ degrees.

\begin{thm} \label{thm:cute}
There is a bijection $\varphi$ between the set of semistandard tableaux of shape $\nu \subseteq (n^k)$ and type  $\{1^{n-1},\ldots, k^{n-1}\}$  and the set of standard tableaux with entries $\{1,2,\dots, k\}$ and shape $(\nu^c)^{\tt t}$.
\end{thm}
\begin{proof}
For a semistandard tableau $\T$ of shape $\nu$ with type $\{1^{n-1},\ldots, k^{n-1}\}$, the standard tableau $\varphi(\T)$ is obtained by the following procedure:  (i)  place $\T$ inside the rectangle of shape  $(n^k)$;  (ii) fill in the rest of the boxes in each column with the missing numbers in $\{1,\dots, k\}$
in ascending order  from bottom to top;  (iii) rotate the resulting tableau 180 degrees and then
take its conjugate.  For example, when $n=3$, $k = 7$,  and
$$
\mathsf T=\grim{0.8}{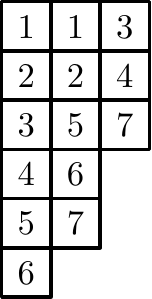} \;,
$$ we obtain
$$
\grim{0.8}{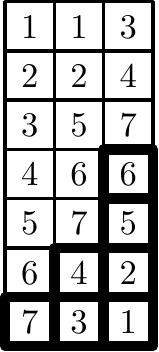} \; ,
$$
so that
$$
\varphi(\mathsf T)= \grim{0.8}{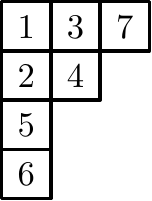}^{\,\,\mathsf t} = \  \grim{0.8}{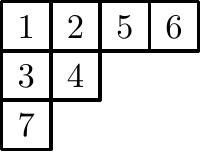}  \;,
$$
which has shape $(\nu^c)^{\tt t}$. The map $\varphi$ is clearly a bijection. \end{proof}

\begin{remark} The theorem above has a number of applications; for example, it  can be used to prove results about representations.  We illustrate this with one quick example.   Recall that the finite-dimensional  irreducible representations $\VV(\lambda)$  of the complex special linear Lie algebra $\fsl_k$ are indexed by partitions $\lambda$
with at most $k$ parts (i.e. $k$ rows).   The irreducible representations $\mathsf{S}_k^\mu$ of the symmetric group $\mathsf S_k$ are labeled by partitions $\mu$ of $k$ and have a basis  indexed by the
standard tableaux of shape $\mu$ with entries $\{1,\dots, k\}$.
Now $\VV(\lambda)$ decomposes into weight spaces (common eigenspaces)
relative to the diagonal matrices in $\fsl_k$,  and the weight space corresponding to
$a_1 \varepsilon_1 + \cdots + a_k \varepsilon_k$  (where $\varepsilon_i$ projects a diagonal
matrix onto its $(i,i)$-entry)   has a basis indexed by the semistandard tableaux of
shape $\lambda$ and type
$\{1^{a_1}, 2^{a_2}, \dots, k^{a_k}\}$.   Because matrices in $\fsl_k$ have  zero trace, $\ve_1+\cdots + \ve_k = 0$ on $\fsl_k$,  and the weight 0  corresponds to
the semistandard tableaux of type $\{1^a,2^a, \dots, k^a\}$, where each entry occurs
with the same  multiplicity.  Thus,   if $0$ is a weight of $\VV(\lambda)$, then $\lambda \vdash ak$ is a multiple of $k$.  The converse holds as well, which can be seen by filling the boxes of $\lambda  \vdash ak$  with
$\underbrace{1,\dots,1}_a,\underbrace{2, \dots, 2}_a, \dots, \underbrace{k,\dots, k}_a$ in order,  starting with the first row
of $\lambda$ and  moving left to  right across the rows from top to bottom.     In particular, if  $\lambda$ is a partition of $(n-1)k$,
for some $n \geq 2$,  and  $\lambda \subseteq (n^k)$, then $\lambda^c$ is a partition of $k$.
The zero weight space $\VV(\lambda)_0$ of $\VV(\lambda)$ has a basis indexed by the semistandard tableaux of type $\{1^{n-1}, 2^{n-1}, \dots, k^{n-1}\}$.  Since the standard tableaux of shape
$(\lambda^c)^{\tt t}$ are in one-to-one correspondence with the standard tableaux of shape $\lambda^c$,
which give a basis for
the irreducible $\mathsf{S}_k$-module $\mathsf{S}_k^{\lambda^c}$,  we have the following consequence of the bijection in Theorem~\ref{thm:cute} under these assumptions:
\begin{cor}  $\dimm \VV(\lambda)_0 =  \dimm \mathsf{S}_k^{\lambda^c} =  \frac{k!}{h(\lambda^c)}$, where
${h(\lambda^c)}$ is the product of the hook-lengths of the boxes in $\lambda^c$.
\end{cor}

The last equality comes from the well-known hook-length formula that gives dimensions of the irreducible $\SF_k$-modules
 (see e.g.  \cite[Thm.~3.10.2]{S}).

\end{remark}

\subsection{Semistandard tableaux  and $\Atwo$-webs}
The bijection we define consists of  a sequence of elementary steps, which we  illustrate by a running example.
\bigskip

\noindent \textbf{Step 0.   Start with a   semistandard tableau $\mathsf{T}$ of shape $(3^n)$ and type \break  $\boldsymbol{\{1^2,\ldots, k^2, k+1, \ldots, 3n-k\}}$:}
\smallskip

\noindent for example, let  $k=7, n=8$,  and
$$\mathsf{T}=\grim{0.8}{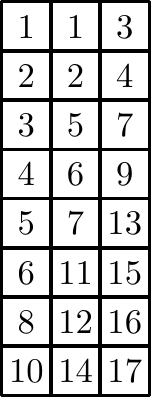}\, .$$

\noindent \textbf{Step 1.  Let  $\mathsf{T}_1$ be the  subtableau of $\mathsf{T}$ having the entries    $\{1^2,\ldots,k^2\}$,  and let $\nu = (\nu_1 \geq \nu_2 \geq \ldots) \subseteq (3^k)$ be the shape of $\mathsf{T_1}$:}
$$
\mathsf{T}_1=\grim{0.8}{exam1.pdf}  \quad   \hbox{\rm with} \ \ \ \nu =  (3^3, 2^2, 1).
$$
\noindent \textbf{Step 2.  Let $\T^+$ be the standard tableau $\varphi(\mathsf{T}_1)$ obtained from $\mathsf{T}_1$ using
Theorem~\ref{thm:cute}:}
$$
\mathsf{T}^+=\grim{0.8}{exam4.pdf}
$$
\noindent \textbf{Step 3. Let  $\mathsf{T}_2$ be the skew-subtableau  of $\mathsf{T}$ formed by the entries $\boldsymbol{\{k+1, \ldots, 3n-k\}}$:}
$$
\mathsf{T}_2= \grim{0.8}{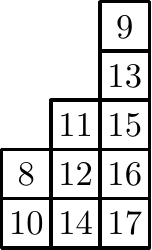}
$$
\noindent \textbf{Step 4.  Replace each entry $i$ of $\mathsf{T}_2$  with $\boldsymbol{3n-k+1-i}$,  rotate the result by 180 degrees,  and take its conjugate
to obtain a standard tableau  $\mathsf{T}^-$:}
$$
\mathsf{T}^ -=
\grim{0.8}{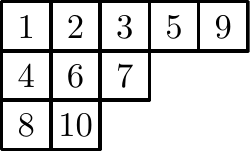}
$$
\begin{itemize}
\item{}{\it $\mathsf{T}^+$ has shape $\lambda: = \left(k-(\nu^{\tt t})_3, \,k-(\nu^{\tt t})_2, \,k-(\nu^{\tt t})_1\right) \vdash k$.}
\item{}{\it $\mathsf{T}^-$ has shape $\mu:=\left(n-(\nu^{\tt t})_3,\, n-(\nu^{\tt t})_2,\,n-(\nu^{\tt t})_1\right) \vdash  3n-2k.$}
 \end{itemize}
  In our example,  $\lambda = (4,2,1)$ and $\mu = (5,3,2)$. \bigskip

\noindent \textbf{Step 5. Assign a pair $(\mathsf{T}^+_\Box,\mathsf{T}^+_\diamond )$ to $\mathsf{T}^+$  by the following procedure}:

\noindent Let  $\ell_1< \cdots < \ell_{\lambda_3}$ be the entries in the third row of $\mathsf{T}^+$.  Do the  following steps in order to create the tableaux  $\mathsf{T}^+_\Box$ and  $\mathsf{T}^+_\diamond$:
\begin{enumerate}[{(I)}]
\item   Place $\ell_1 \ldots, \ell_{\lambda_3}$ from left to right across the third row of $\T^+_\Box$.
\item  For $p=1, \ldots, \lambda_3$,   let $j_p$ be  the largest entry in the second row  of $\T^+$ less than $\ell_p$, such that  $j_p$ is different from  $j_1, \ldots, j_{p-1}$ when $p > 1$.
 Place  $j_1, \ldots, j_{\lambda_3}$ in the second row of $\T^+_\Box$ so that they are
increasing from left to right across the row.
\item  For $p=1, \ldots, \lambda_3$,   let $i_p$ be  the largest entry in the first  row  of $\T^+$ less than $j_p$, such that  $i_p$ is different from  $i_1, \ldots,  i_{p-1}$ when $p > 1$.
 Place  $i_1, \ldots,  i_{\lambda_3}$  in the first row of $\T^+_\Box$ so that they are
increasing from left to right across the row.
\item  Shift the  remaining entries in each row of  the tableau $\T^+$ to the left to form the tableau  $\mathsf{T}^+_\diamond$.
\end{enumerate}

The shape of
$\T^+_\Box$ is  $(\lambda_3,\lambda_3,\lambda_3)$, and the shape of $\mathsf{T}^+_\diamond$ is the two-rowed partition
$(\lambda_1-\lambda_3, \lambda_2-\lambda_3) =
\left( (\nu^{\tt t})_1-(\nu^{\tt t})_3,  (\nu^{\tt t})_1-(\nu^{\tt t})_2 \right)$.
$$
\mathsf{T}^+_\Box=\grim{0.8}{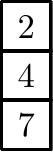} \; , \qquad
\mathsf{T}^+_\diamond=\grim{0.8}{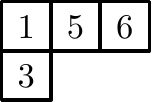} \; .
$$

 By its construction, $\T^+_\Box$ is a standard tableau.  To verify that $\T^+_\diamond$ is standard, it suffices to show that
  if $a$  is the $(1,s)$-entry and $b$ is the $(2,t)$-entry  of $\T^+$, and they are in the same column of $\T^+_\diamond$, then $s \le t$.    Assume the contrary,  that $s>t$.
Since $a$ is immediately above $b$  in the same column of $\T^+_\diamond$ and $s > t$,  there must be a pair
$(i_p,j_p)$ from \textbf{Step 5} above with $i_p$ to the left of $a$ and $j_p$ to the right of $b$ in $\T^+$,  such that
$a < j_p$.   However, this contradicts the way  $i_p$ was chosen in \textbf{Step 5}.
So, we conclude that $\T^+_\diamond$ is a standard tableau.

 \medskip
\noindent \textbf{Step 6.   Perform (2) of Algorithm  \ref{alg:tym}  with $\T^+_\Box$ to get an $\EM$-diagram $\M^+$  and then trivalize  the vertices to obtain an $\Atwo$-web  $\W^+$ whose boundary has $k$ ``$+$"s  on the top:}
$$
\W^+ = \grim{0.6}{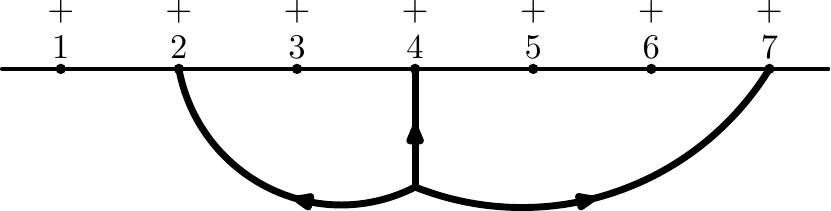}
$$

 \medskip
\noindent \textbf{Step 7. Apply Step 5  to  $\T^-$, to  obtain  $\T^-_\Box$ and $\mathsf{T}^-_\diamond$:}
$$
\T^-_\Box=\grim{0.8}{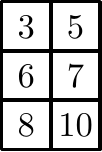} \; , \qquad
\T^-_\diamond=\grim{0.8}{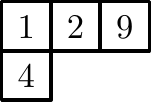}
$$
{\it The shape of  $\T^-_\Box$ is  $(\mu_3,\mu_3,\mu_3)$ and of  $\mathsf{T}^-_\diamond$  is
$(\mu_1-\mu_3, \mu_2-\mu_3)$ which is $\left( (\nu^{\tt t})_1-(\nu^{\tt t})_3,  (\nu^{\tt t})_1-(\nu^{\tt t})_2 \right)$. Hence, $\mathsf{T}^+_\diamond$ and $\mathsf{T}^-_\diamond$ have the same shape.
}
 \bigskip

\noindent \textbf{Step 8. Apply Step 6 to  $\T^-_\Box$ to obtain an $\EM$-diagram $\M^-$ and an $\Atwo$-web  $\W^-$  whose boundary has $3n-2k$ ``$-$"s  on the bottom: }
$$
\W^- = \grim{0.6}{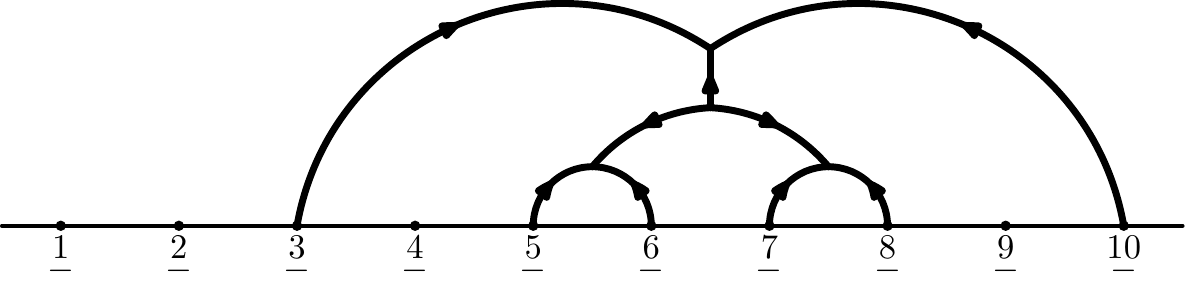}
$$

\noindent \textbf{Step 9.  If  $\{\alpha_1 < \dots < \alpha_\ell\}$ are the entries  of  $\mathsf{T}^+_\diamond$, and $\{\beta_1 < \dots < \beta_\ell\}$ are the entries of  $\mathsf{T}^-_\diamond$,  replace each $\alpha_i$ and each $\beta_i$
by $i$ to produce a pair of standard tableaux with entries in $\{1,\dots, \ell\}$.   Then perform the inverse of the Robinson-Schensted correspondence on the pair  to get a permutation $\sigma$ in $\mathsf{S}_\ell$:}
$$
\left( \mathsf{T}^+_\diamond=\grim{0.8}{exam6.pdf}\;, \quad
\mathsf{T}^-_\diamond=\grim{0.8}{exam10.pdf}\right)
\longrightarrow
\left(\grim{0.8}{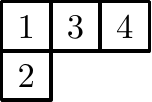} \;,
\quad
\grim{0.8}{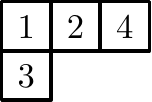} \right)
$$
$$
(\R,\Q)  \ = \ \left(\grim{0.8}{exam11.pdf} \;, \quad \grim{0.8}{exam12.pdf} \right)
 \longrightarrow
\sigma=
\begin{pmatrix}
1 &2 &3 &4 \\
3 &1 &2 &4
\end{pmatrix}.
$$

\medskip
\noindent \textbf{Step 10.   For each $i=1, \ldots, \ell$,  connect  vertex $\alpha_i$ on the  top row  to vertex $\beta_{\sigma(i)}$ on bottom row (which we refer to as \emph{isolated vertices}) according to the following rule:}
\medskip

\emph{The path from $\alpha_i$ to $\beta_{\sigma(i)}$ should stay within the open region
$\mathcal{R}_o$ between the two webs  $\W^+$ and $\W^-$ and cross the smallest number of arcs  in
$\W^+$ and $\W^-$.
After these paths are drawn, let $\Phi_{\ast}(\T)$ be the resulting diagram.}
$$
 \grim{0.6}{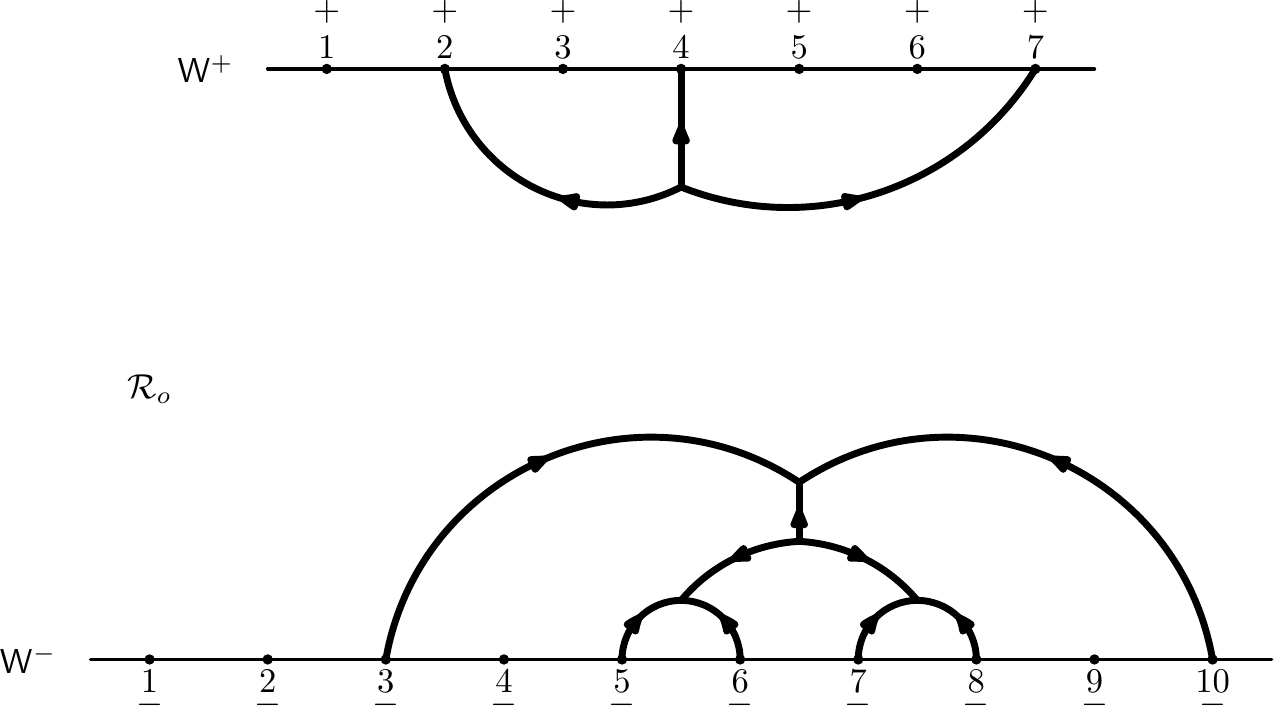}
$$
In the example,  $\sigma = \begin{pmatrix} 1 &2 &3 &4 \\ 3 &1 &2 &4 \end{pmatrix}$,
so we connect  the isolated vertex $\alpha_1=1$ on the top with the isolated vertex $\beta_3=4$ on the bottom,
 $\alpha_2=3$ on the top with $\beta_1=1$ on the bottom,  $\alpha_3=5$ on the top with $\beta_2=2$ on the bottom and
  $\alpha_4=6$ on the top with $\beta_4=9$ on the bottom.

$$
\Phi_{\ast}(\T) \ = \ \grim{0.6}{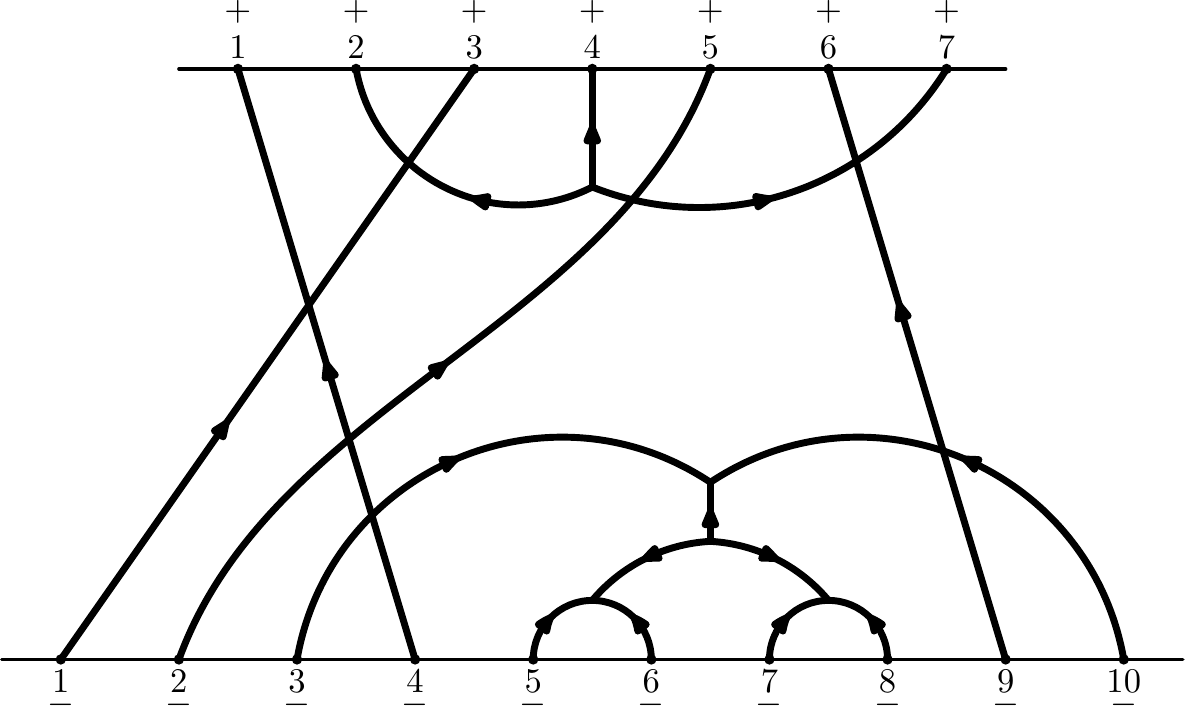}
$$

\noindent \textbf{Step 11.
Trivalize the vertices  as described in Section~\ref{sec:intro} to produce an $\Atwo$-web $\Phi(\mathsf{T})$:}
$$
\Phi(\T) \ = \  \grim{0.6}{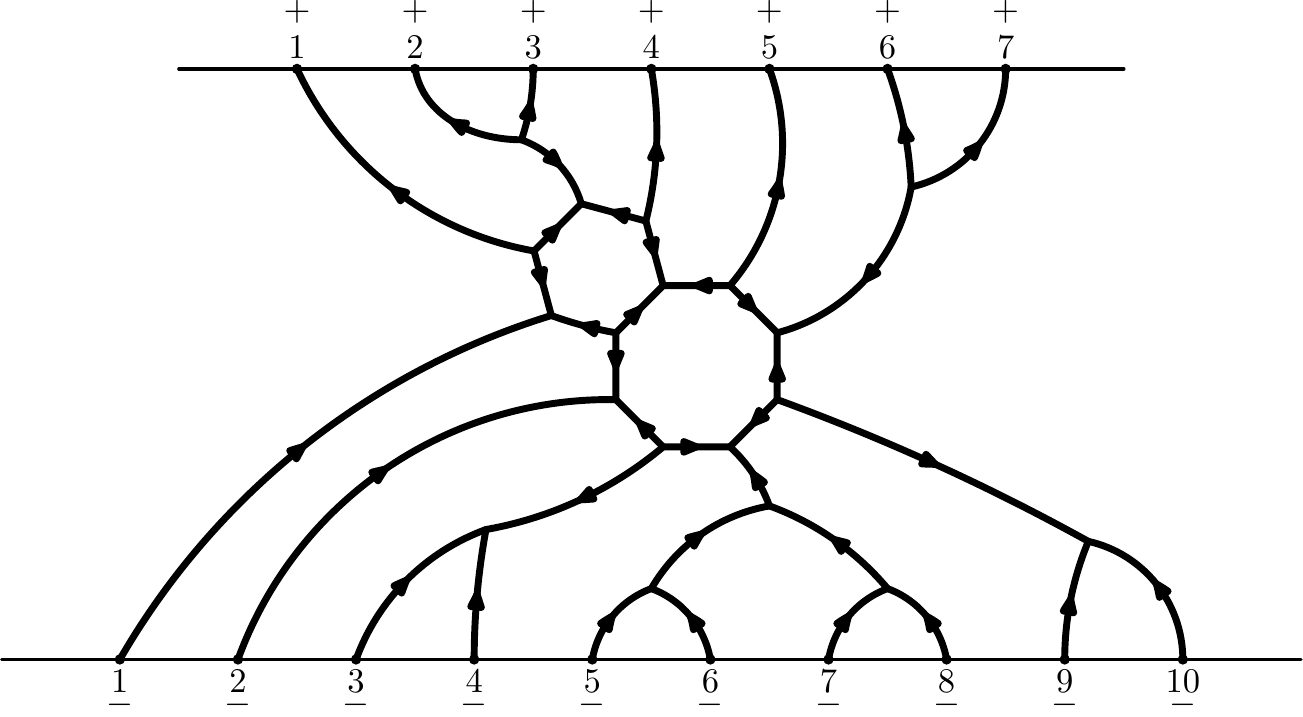}
$$

Our aim is to show the following:
\begin{thm} \label{thm:main} The map $\T \mapsto \Phi(\T)$ is a bijection between the set of
semistandard tableaux $\T$ of shape $(3^n)$ and type $\{
1^2, \dots, k^2, k+1, \dots, 3n-k\}$ and the set of  nonelliptic $\Atwo$-webs
with $k$  ``$+$''s on the top boundary and  $3n-2k$  ``$-$''s on the bottom boundary. \end{thm}
We close this section with  some comments and a result we will use  later.

Recall that a permutation $\sigma \in \SF_\ell$ is $3\,2\,1$-avoiding if there do not exist
$1 \leq a < b < c  \leq \ell$ such that $\sigma(c) < \sigma(b) < \sigma(a)$.
It is well known that the  RS correspondence gives a bijection
between the set of $3\,2\,1$-avoiding permutations in $\SF_\ell$ and  the set of pairs of tableaux of shape $\lambda$
for some partition $\lambda$ of $\ell$,  whose length is less than or equal to $2$ (see  for example, \cite{Wi}).   Because the tableaux
$\mathsf{T}^+_\diamond$ and $\mathsf{T}^-_\diamond$ above are of the same shape and have at most 2 rows,
the permutation obtained in Step 9 will be $3\,2\,1$-avoiding.   The steps above construct an
$\Atwo$-web having   $k$ ``$+$''s on the top boundary and $ 3n-2k$ ``$-$''s on the bottom boundary.
The top and bottom boundaries have the same number of isolated vertices, and the
$3\,2\,1$-avoiding permutation $\sigma$  prescribes how to connect  the isolated vertices on the top boundary to
those on the bottom boundary.  Therefore, we have the following corollary:

\begin{cor} Assume $\T$ is a semistandard tableau as in Theorem \ref{thm:main}.   Then Steps 1-10 above
associate to $\T$ a triple $(\M^+,\M^-, \sigma)$, where $\M^+$ is an $\EM$-diagram with $k$ ``$+$''s on the top boundary,
$\M^-$ is an $\EM$-diagram with $3n-2k$ ``$-$''s on the bottom boundary, and $\sigma$ is a $3\,2\,1$-avoiding permutation.
\end{cor}

The notions  of  circle depth and path depth from  \cite{T} will enable us to prove Theorem \ref{thm:main}.
The faces (regions) determined by an $\EM$-diagram $\M$ and its trivalization $\W$  are the same.
 \begin{definition}\label{def:circle}  Suppose $\M$ is an $\EM$-diagram that lies above a horizontal boundary line,  and let
 $\W$  be its trivalization.   Let  $x$ be a point lying on a face of $\M$ (or $\W$).
\begin{itemize}
\item The {\em circle depth} of $x$, denoted $d_{\mathsf c}(x)$,
is the number of semicircles in  $\M$  that lie above $x$.
\item The {\em path depth} of $x$, denoted $d_{\mathsf p}(x)$, is the minimal number of edges in $\W$ crossed by any path from $x$ to
the unbounded region $f_\infty$ that lies above the boundary line.
\end{itemize}
 In particular, any point in $f_\infty$ has circle depth and
path depth equal to $0$.
 \end{definition}

The following result  can be found in \cite{T}.
\begin{prop} \label{depth}{\rm (\cite[Lemma 4.5]{T})}
Let $\SF$ be a standard tableau of shape $(3^n)$,  and let  $\M$ be the $\EM$-diagram and $\W$ the $\Atwo$-web  corresponding to $\SF$ by
Algorithm \ref{alg:tym}. Then, for  a point $x$ in a face of $\M$ (or $\W$), the path depth and the circle depth of $x$  agree.
\end{prop}

In the following example, the numbers indicate path  (circle) depths,
which are the same for all interior points in a given face.

$$
\grim{0.6}{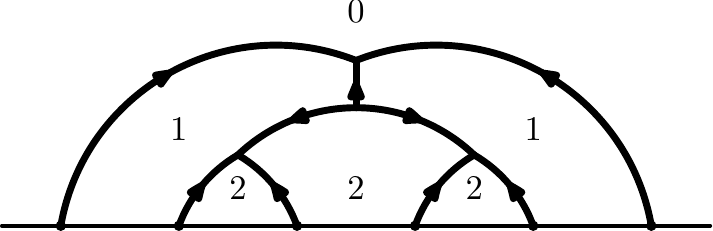}
$$

\begin{remark} Because of this result, the only time we will need to make a distinction between an $\EM$-diagram  and its trivalization
is when we show the nonelliptic property.
\end{remark}

\section{$\Phi(\T)$ is a nonelliptic $\Atwo$-web}\label{sec:proof}
We fix two positive integers $k$ and $n$ such that $2k\leq 3n$ for
all arguments in what follows.  Let   $\T$ be a semistandard tableau of shape $(3^n)$ and type $\{
1^2, \dots, k^2,\allowbreak k+1, \dots, 3n-k\}$,  and assume  $\Phi(\T)$ is the
$\Atwo$-web obtained by applying the algorithm given in the previous
section. 
Although the process certainly defines an $\Atwo$-web, it is not  readily apparent
that $\Phi(\T)$ is \emph{nonelliptic}. We give a proof
that $\Phi(\T)$ is nonelliptic  in this section.

We will use the notation from the steps in Section~\ref{sec:bijection} that created $\Phi(\T)$.
In particular,  $\M^+$ and $\M^-$ are the $\EM$-diagrams produced from 3-rowed standard tableaux in Steps 6 and 8, respectively.

In Algorithm~\ref {alg:tym}, an entry $\ell$ in the third row of a tableau is connected to an entry $j$ in the second row,  and $j$ is connected to an entry  $i$ in the first row. We refer to the union of the arcs connecting $\ell$  with  $j$ and $j$  with $i$ as an \emph{$\EM$-configuration}.  Thus, the $\EM$-diagrams considered here are collections of $\EM$-configurations.

 \begin{lemma}\label{lemma:relative}{\rm (\cite{T})} If $\EM^1$ and $\EM^2$ are two
$\EM$-configurations in an $\EM$-diagram obtained from a
standard tableau with three rows,  then there are only five possible relative
positions for $\EM^1$ and $\EM^2$ as depicted in the
following:
\end{lemma}
$$
\text{I: } \grim{0.7}{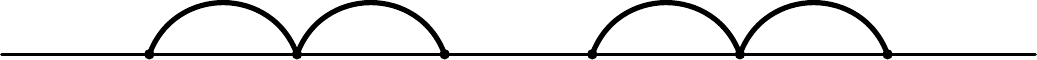}
$$
$$
\text{II: } \grim{0.7}{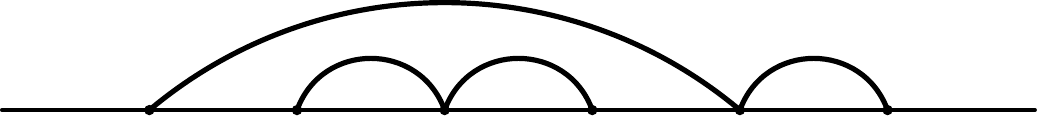}
$$
$$
\text{III: } \grim{0.7}{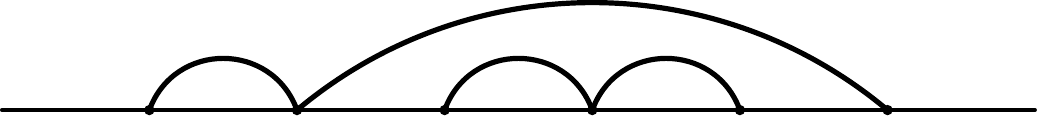}
$$
$$
\text{IV: } \grim{0.7}{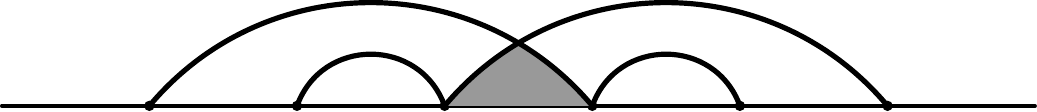}
$$
$$
\text{V: } \grim{0.7}{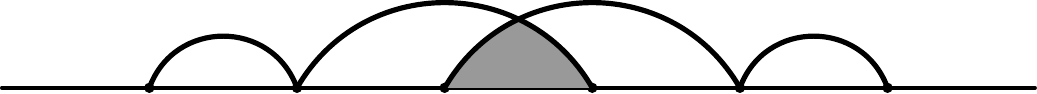}
$$
\begin{proof}
There are five standard tableaux of shape $(2^3)$,
and each standard tableau corresponds to one of the relative positions of two
$\EM$-configurations.
\end{proof}

Any $\EM$-configuration $\mathsf{m}$ consists of a left and a right arc (semicircle), which
we denote $\EM_L$ and $\EM_R$.  Now assume $\T^+$ and $\T^-$ are the tableaux constructed in Steps 2 and 4 and $\M^+$ and $\M^-$ are the corresponding $\EM$-diagrams from Steps 6 and 8, respectively.

\begin{lemma}\label{lemma:isopos}
Assume $v$ is an isolated vertex in $\M^-$ (resp. in $\M^+$), and $v$ has label $a$. Let $\EM$ be an  $\EM$-configuration in $\M^-$ (resp. in $\M^+$). If $v$ is in   $\EM_R$, then $a$ lies in the first row of $\T^-$ (resp. of  $\T^+$).  If  $v$ is in   $\EM_L$, then $a$ lies in the second row of $\T^-$ (resp. of  $\T^+$).
\end{lemma}
\begin{proof}
 Assume $\EM$ has  vertices labeled  $i<j<\ell$.  Since $v$ is an isolated vertex, $a$ lies in either the first or second row of $\T^-$.
If $v$ is in  $\EM_R$,  then  $a$ is greater than $j$, so
it must lie in the first row of $\T^-$ by the way $j$ was chosen. Similarly, if $v$ is in  $\EM_L$, then  $a$ is greater than $i$ and less than $j$,
so it must lie in the second row of $\T^-$ by the way $i$ was chosen.
\end{proof}

\begin{cor}\label{lemma:LR}
Let $v$ be an isolated vertex in $\M^-$ (resp. in $\M^+$), then it is impossible for $v$ to be in both $\EM^1_L$ and $\EM^2_R$ for two $\EM$-configurations $\EM^1$ and $\EM^2$ of $\M^-$ (resp. of $\M^+$).
\end{cor}

\begin{cor} There exists a unique shortest path from an isolated
vertex $v$ on the bottom boundary  to the unbounded region $\mathcal{R}_o$ of  $\Phi_*(\T)$.
\end{cor}

\begin{proof} Due to Proposition~\ref{depth}, there exists a
shortest path from $v$ to $\mathcal{R}_o$, and such a path is determined
by the circle depth of $v$. The path moves from one face of $\M^-$  to another so
that the circle depth is strictly decreased. Therefore, if there is
only one choice of direction for moving, the shortest path must be
unique.  There is only one choice by
Corollary~\ref{lemma:LR}, because there can be no isolated vertices
in the shaded regions of figures IV and V of Lemma \ref{lemma:relative},  since those regions lie inside
both a right semicircle and a left semicircle.
\end{proof}

Recall from \textbf{Step 9} in Section~\ref{sec:bijection} that $\{\alpha_1 < \dots < \alpha_\ell\}$ are the labels of the isolated vertices in $\M^+$,
$\{\beta_1 < \dots < \beta_\ell\}$ are the labels of the isolated vertices in $\M^-$, and
$\sigma$ is a  $3\,2\,1$-avoiding permutation on $\{1, \ldots, \ell\}$.
Let $\mathsf{i}:  \{\alpha_1 , \dots, \alpha_\ell\} \longrightarrow  \{1, \ldots, \ell\}$ and
${\mathsf j}:  \{\beta_1 , \dots , \beta_\ell\} \longrightarrow  \{1, \ldots, \ell\}$ be order-preserving bijections.

\begin{lemma}\label{lemma:nocrossing1}
Let $v_1$ and $v_2$ be two isolated vertices labeled by $a$ and $b$
respectively,  where $a<b$.   If  $v_1$ and $v_2$ satisfy either of the following conditions,
\begin{itemize}
\item[{\rm (i)}]   $v_1$ and $v_2$ are contained in a semicircle in $\M^+$ (resp. $\M^-$),
\item[{\rm (ii)}]  $v_1$ is inside $\EM^1_\LL$  and $v_2$ is inside $\EM^2_\R$  for some $\EM$-configurations $\EM^1$ and $\EM^2$ of $\M^+$ (resp. $\M^-$),
\end{itemize}
then $\sigma\left(\mathsf{i}(a)\right)<\sigma\left(\mathsf{i}(b)\right)$
(resp. $\sigma^{-1}\left(\mathsf{j}(a)\right) < \sigma^{-1}\left(\mathsf{j}(b)\right)$).
\end{lemma}

\begin{proof} For both (i) and (ii), we  present the proof only for the case $\M^+$.

(i)  Assume  $v_1$ and $v_2$ lie  inside one semicircle.  Then  they are in either
$\EM_L$ or $\EM_R$ of  an $\EM$-configuration of $\M^+$. Therefore,  $a$ and $b$
must be on the same row of $\T^+_\diamond$.

Suppose that $a$ and $b$ are in the second row of $\T^+_\diamond$,  and
suppose on the contrary that $\sigma\left(\mathsf{i}(a)\right)>\sigma\left(\mathsf{i}(b)\right)$.
Consider the  RS correspondence applied to
$\sigma=\begin{pmatrix} 1 & \ldots & \ell \\ \sigma(1) & \ldots & \sigma(\ell) \end{pmatrix}$.
Since $a$ is in the second row of $\T^+_\diamond$, there is a bump when we insert $\sigma(\mathsf{i}(a))$. So there exists  an $x$ with  $x<\mathsf{i}(a)$, $\sigma(x) > \sigma\left(\mathsf{i}(a)\right)$.
Thus, we have $x<\mathsf{i}(a)<\mathsf{i}(b)$ and $\sigma(x)> \sigma\left(\mathsf{i}(a) \right) >\sigma(\left(\mathsf{i}(b)\right)$.
This contradicts the fact that
$\sigma$ is $3\,2\,1$-avoiding.

Suppose that $a$ and $b$ are in the first row of $\T^+_\diamond$,
that is,  $v_1$ and $v_2$ are in the right semicircle $\EM_R$. Let $i<j<k$ be the three vertex labels of $\EM$,  so that  $j<a<b<k$.
We again assume that $\sigma\left(\mathsf{i}(a)\right)>\sigma\left(\mathsf{i}(b)\right)$.
Note that, when we insert $\sigma\left(\mathsf{i}(a)\right)$,  there is no bump, since $a$ is in the first row of $\T^+_\diamond$. Then, since $b$ is in the first row, there is no bump when we  insert $\sigma\left(\mathsf{i}(b)\right)$ also. But $\sigma\left(\mathsf{i}(a)\right)>\sigma\left(\mathsf{i}(b)\right)$.
Hence, $\sigma\left(\mathsf{i}(a)\right)$ must be bumped down to second row before inserting $\sigma\left(\mathsf{i}(b)\right)$. Therefore, there exists $x$ with $\mathsf{i}(a)<x<\mathsf{i}(b)$ and $\mathsf{i}^{-1}(x)$ is in the second row of $\T^+_\diamond$. This is impossible by
the way $j$ was chosen since $j < x <k$.

(ii)   If  $v_1$ is inside $\EM^1_L$  and $v_2$ is inside
$\EM^2_R$,  then $a$ is in the second row of
$\T^+_\diamond$ and $b$ is in the first row of $\T^+_\diamond$.
Suppose on the contrary that $\sigma\left(\mathsf{i}(a)\right)>
\sigma\left(\mathsf{i}(b)\right)$. Since $a$ is in the second row, there is a bump when we insert $\sigma\left(\mathsf{i}(a)\right)$ while applying the RS correspondence to $\sigma$.
So there exists an $x$ with $x<\mathsf{i}(a)$ and $\sigma(x) > \sigma\left(\mathsf{i}(a)\right)$.
Then, $x <\mathsf{i}(a)<\mathsf{i}(b)$ and $\sigma(x)>\sigma\left(\mathsf{i}(a)\right)>\sigma\left(\mathsf{i}(b)\right)$.
This contradicts the  $3\,2\,1$-avoiding property of $\sigma$.  \end{proof}

\begin{lemma}\label{lemma:trivalizing} Let $R$ be a bounded region in $\Phi_\ast(\T)$,
where $\Phi_\ast (\T)$ is  as in \textbf{Step 10} of Section~\ref{sec:bijection}.
Suppose that there is a
point of crossing (which has not yet been trivalized)  on the boundary of $R$ such that a directed edge on the boundary is moving toward the point, and the other directed
edge on the boundary is moving away  from the point. Then when we
trivalize the crossing,  there will be one more boundary edge added
to the boundary of $R$.
\end{lemma}
\begin{proof}
This is clear from the definition of  trivalizing:
\begin{equation}\label{eq:addaline}
 \grim{0.55}{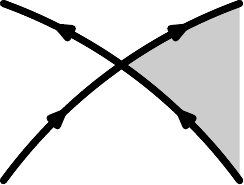} \; \longrightarrow \; \grim{0.55}{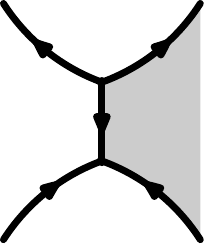}\,.
\end{equation}
\end{proof}

The following proposition completes the proof that $\Phi(\T)$ in Theorem~\ref{thm:main} is a nonelliptic $\Atwo$-web.

\begin{prop}
For a semistandard tableau $\T$ of shape $(3^n)$ and type
$\{1^2,\ldots, k^2, k+1, \ldots, 3n-k\}$, the $\Atwo$-web $\Phi(\T)$ is
\emph{nonelliptic}.
\end{prop}
\begin{proof}
Let $R$ be a bounded,  internal  region of $\Phi_\ast(\T)$, and let
$\mathcal{R}_o$ be the open region between the $\Atwo$-webs  $\W^+$ and $\W^-$ of
$\Phi_\ast(\T)$,  as in \textbf{Step 10} of Section~\ref{sec:bijection}.
After trivalization has taken place,  internal regions  are bounded  by an even number of edges, since \emph{sink} and \emph{source} vertices alternate along the boundary of internal faces. Hence,  if $R$ has  more than four boundary edges in $\Phi_\ast(\T)$, then after trivalization,  $R$ is bounded by at least six edges.  So we may restrict our attention to
the case that $R$ is bounded by either  three or four edges.
We know from the work of Tymoczko \cite{T}, that $\W^+$ and
$\W^-$ (as $\Atwo$-webs) have no internal regions with four edges.
Therefore, we may suppose that $R$ has at least one edge coming
from a vertical line.  At this juncture, we  break our considerations into several cases.
\medskip

\verb"Case 1: " {\it $R$ is not contained in $\mathcal{R}_o$}

We treat only the $\W^-$ case, since every
argument can be translated to $\W^+$ by reversing the arrows on
the directed edges.

Since there are no crossings between vertical lines outside
$\mathcal{R}_o$, there must be a part of  the vertical line passing two
semicircular arcs  of $\W^-$; let them be $\EM^1_\bullet$ and
$\EM^2_\bullet$ so that the vertical  line passes $\EM^1$  first in going from bottom to top,
where $\bullet$ is either $\mathsf{R}$ or $\mathsf{L}$. By Corollary~\ref{lemma:LR} and
Lemma~\ref{lemma:relative}, we see that $\EM^1_\bullet$ and
$\EM^2_\bullet$ do not intersect each other. $$
\grim{.6}{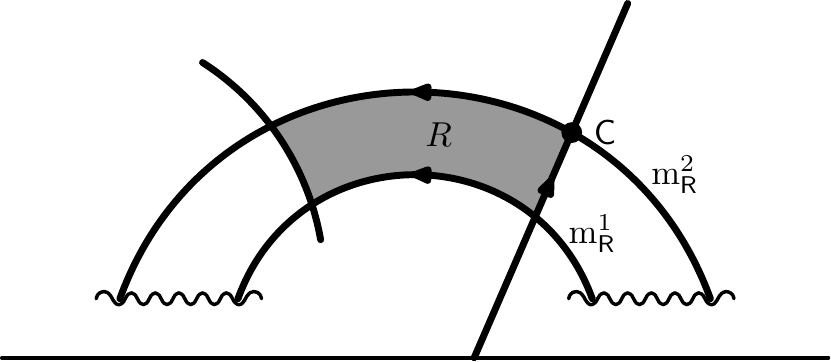}  \qquad \grim{.6}{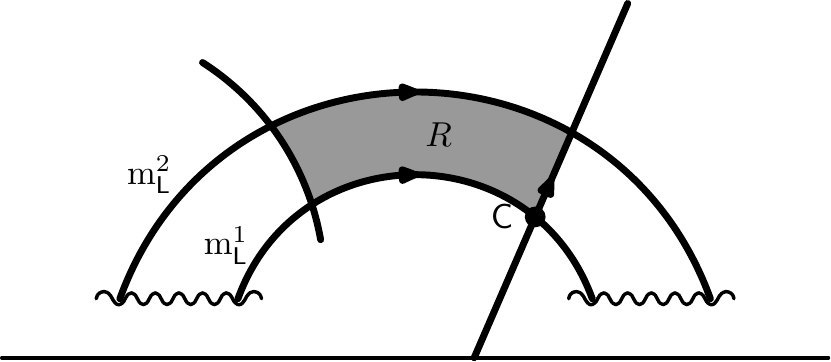}
$$
Hence,  there must be another arc (part of a vertical line or a semicircle).
Since the direction of every arc in $\W^-$ is upward,  before
trivalization  there will be at least one point (of crossing) such
that one edge on the boundary of $R$ is moving from the point while
the other is toward the point. This crossing point is indicated by  $\mathsf{C}$ in the illustrations above. Therefore, as in \eqref{eq:addaline}, at least one more boundary edge will be added to $R$ upon trivalizing, so there
will be at least six edges on the boundary of $R$ in $\Phi(\T)$.

\medskip

\verb"Case 2: " {\it  $R$ is contained in $\mathcal{R}_o$.}
\begin{description}
\item $\bullet$ The boundary edges of $R$ are all from  vertical lines: \\
Since the permutation $\sigma$ in \textbf{Step 9} is $3\,2\,1$-avoiding, there cannot  be a
    bounded region with three sides.  But when  $R$ is bounded by four edges,  then since every vertical line moves upward,   trivalizing crossings will add at
    least one more edge (see \eqref{eq:addaline}).

\item $\bullet$  There are boundary edges of $R$ from just one of the webs $\W^+$ or $\W^-$: \\
     Assume that $R$ has an boundary edge  $\mathfrak a$  from $\W^+$.   Then
    $\mathfrak{a}$ must be a connected arc.    If $\mathfrak{a}$   is a part of a semicircle,  then there must be two vertical lines  on the boundary of $R$ that  pass through the semicircle. By Lemma~\ref{lemma:nocrossing1}, these two vertical lines do not make a crossing,  and there must be at least one
    more vertical line on the boundary of $R$.
 If there are at least two more vertical lines,  then $R$ will clearly have
at least five edges  (four from the four vertical lines and the edge $\mathfrak{a}$), and after trivalizing  the crossings, $R$ has  at least six edges.
 $$
 \grim{.6}{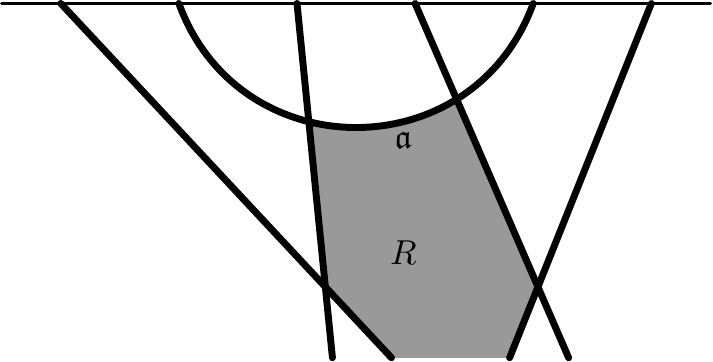}
 $$

    If there is only  one more vertical line on the boundary of
    $R$, then since every line is moving upward, the crossings of
    two lines on the boundary of $R$ will make another edge when we
   trivalize the crossings, due to Lemma~\ref{lemma:trivalizing}.
 $$
 \grim{.6}{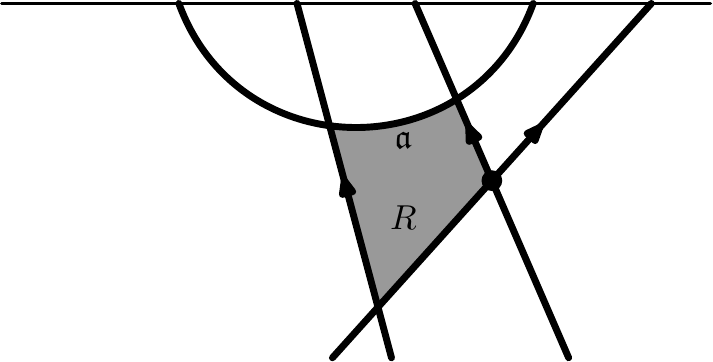}
 $$

    Now we consider the case that there are at least two semicircles whose subarcs are in
    $\mathfrak{a}$.  Assume  $\ell_1$ (resp. $\ell_2$) is a  vertical line  on the boundary of $R$ which is  connected to  a top vertex,  
    and let  $\EM^1_*$ (resp. $\EM^2_\bullet$)  be the  semicircle that $\ell_1$ (resp. $\ell_2$)  passes through.
    If the pair $(*,\bullet)$ is $(\LL,\R)$, then Lemma~\ref{lemma:nocrossing1}  implies that
    $\ell_1$ and $\ell_2$ do not cross each other. Therefore, there must be
   another vertical line on the boundary of $R$,  and by Lemma~\ref{lemma:trivalizing}, there will be at least six
   edges on the boundary of $R$ after trivalizing the crossings.
    $$
    \grim{.6}{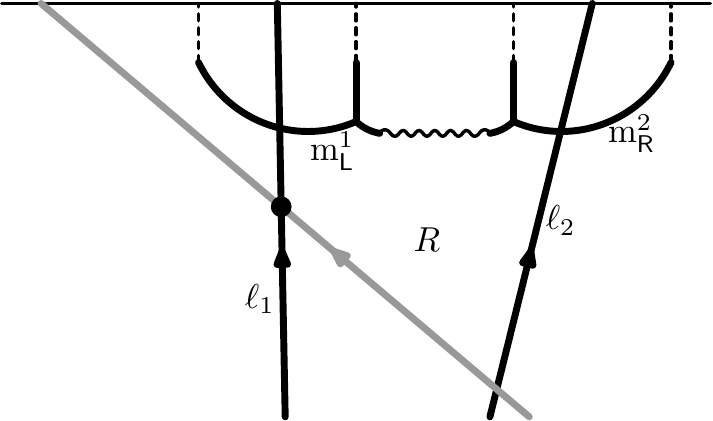}
    $$

  When $(*,\bullet)$ is $(\R,\R)$  as depicted  in (I) below,  trivalizing the intersection point of $\EM^1_\R$ and $\ell_1$
   will produce one more edge on the boundary of $R$ by Lemma~\ref{lemma:trivalizing}.
   When $(*,\bullet)$ is $(\R,\LL)$  as in (II),  trivalizing the crossing of $\EM^1_\R$ and $\ell_1$
   and of $\EM^2_\LL$ and $\ell_2$ will produce two more edges on the boundary of $R$.
   Finally,  when  $(*,\bullet)$ is $(\LL,\LL)$ as in (III), trivalizing the crossing point of $\EM^2_\LL$ and $\ell_2$ will produce one more edge on the boundary of $R$.

\begin{figure}[ht]
\captionsetup[subfigure]{labelformat=empty}
\centering
\begin{subfigure}[b]{0.33\textwidth}
\centering
$\grim{.6}{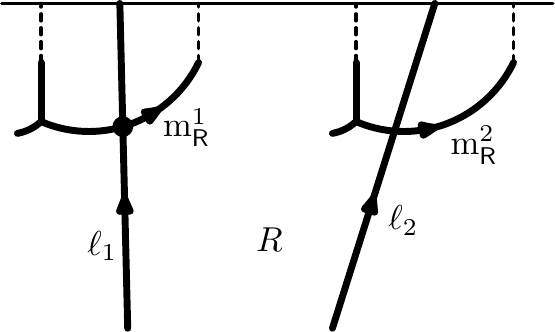}$
\caption{(I)}
 \label{fig:case a}
 \end{subfigure}%
\begin{subfigure}[b]{0.33\textwidth}
\centering
$\grim{.6}{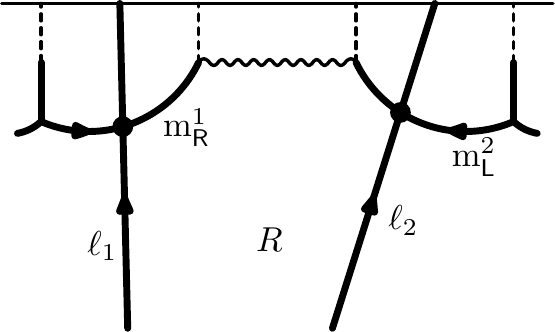}$
\caption{(II)}
 \label{fig:case b}
 \end{subfigure}%
\begin{subfigure}[b]{0.33\textwidth}
\centering
$\grim{.6}{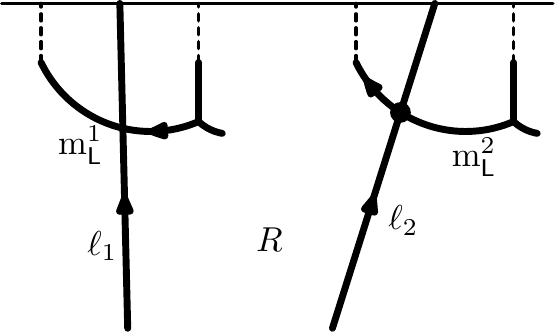}$
\caption{(III)}
\label{fig:case c}
\end{subfigure}%
\end{figure}

 \item $\bullet$ There are boundary edges of $R$ from both $\W^+$ and $\W^-$: \\
    There must be two noncrossing vertical lines on the boundary of $R$.
    Therefore, if at least two semicircles appear on the boundary on
    $\W^+$ or $\W^-$,  then $R$ is not a  quadrilateral.   Hence, we only have to
   treat  the case when just one semicircle from each of   the webs $\W^+$
    and $\W^-$ appears on the boundary:   It is easy to check that there
    are at least two  crossing points  between vertical lines and
    semicircles where we can apply Lemma \ref{lemma:trivalizing},  so
    that at least two more edges will be added after trivalization.
\end{description}   \end{proof}

\section{$3\,2\,1$-avoiding permutations} \label{sec:combinatorics}
In Section~\ref{sec:inverse},
we will use the path depth  of Definition~\ref{def:circle} to  create a pair of standard tableaux from an $\Atwo$-web,  and this will  enable us to show that the map $\Phi$ in Section~\ref{sec:bijection} is a bijection.
The algorithm that constructed the $\Atwo$-web  $\Phi(\T)$ from a semistandard tableau $\T$  produced a
$3\,2\,1$-avoiding permutation via the  RS correspondence in Step 9.   In this section,  we construct a pair of standard tableaux from a $3\,2\,1$-avoiding permutation using  the path depth of a permutation  and show that this pair is exactly the same as the one coming from  the RS correspondence.

The $3\,2\,1$-avoiding permutations in $\SF_\ell$  and the Temperley-Lieb diagrams with $2\ell$ vertices  are both counted by the Catalan number $\frac{1}{\ell+1}\binom{2\ell}{\ell}$ (see for example, \cite[Chap.~6]{St}).   In Remark~\ref{rem:tlbijection}, we show
how the considerations here give a direct bijection between the set
 of $3\,2\,1$-avoiding permutations and the set of  Temperley-Lieb diagrams.    Bijections using the reduced
 expressions for Temperley-Lieb diagrams as words in the  generators are known and can be found for example
 in  \cite{BJS}.

\emph{For the remainder of the section,  we assume that $\sigma \in \SF_\ell$ is a $3\,2\,1$-avoiding permutation.}

\begin{lemma} \label{lem:s5l1}
Assume $1 \le s<t \le \ell$. \begin{itemize}
\item[{\rm (i)}]   If  $\sigma(s) \ge s$ and $\sigma(t) \ge t$, then  $\sigma(s) < \sigma(t)$.
\item[{\rm (ii)}]  If  $\sigma(s) \le s$ and $\sigma(t) \le t$, then  $\sigma(s) < \sigma(t)$.
\item[{\rm (iii)}]  If  $\sigma(s)=s$, then $\sigma(i) <s$ for all $i <s$ and $\sigma(j) >s$ for all $j >s$.
\end{itemize}
\end{lemma}
\begin{proof}
For (i),  assume the contrary that  $\sigma(s) > \sigma(t)\ge t$.
If  for all $i>t$, $\sigma(i)  > \sigma(t)$,  then $\sigma(i)>t$ since $\sigma(t)\ge t$.
Then $\sigma$ is not bijective,  since $\sigma(s) > \sigma(t)\ge t$ and $s<t$.
Thus, for some $i >t > s$, we have $\sigma(i) < \sigma(t) < \sigma(s)$, which contradicts the fact that $\sigma$ is  $3\,2\,1$-avoiding.    Part (ii) can be done by a similar argument,   and (iii) follows from (i) and (ii).
\end{proof}

We associate to $\sigma$ a diagram with vertices numbered $1,\dots, \ell$ from left to right
on the top and on the bottom boundary lines,  and with vertex $i$ on top connected to
vertex $\sigma(i)$ on the bottom  by a vertical line.
The  vertical lines divide the region between the top and bottom boundaries into \emph{faces}.
Next we introduce an analog of the path depth given in Definition~\ref{def:circle} for the diagram of $\sigma$.

\begin{definition}\label{def:path2}{\rm (See \cite{KK}.)}
Let $f_\infty$ be the unbounded face open to left of the diagram of $\sigma$.
For  any point $x$ in the interior of a face of the diagram,  the {\em path depth} of $x$, denoted $d^{\mathsf p}(x,\sigma)$, is the minimal number of edges crossed by any path from $x$ to $f_\infty$ that does not intersect  the top or bottom boundaries.
\end{definition}
Note that when a path from $x$ to $f_\infty$   crosses horizontally at the intersection of  two
lines in the diagram,  it is counted as a single edge crossing.
All points in the interior of a face $f$  have the same path depth, so  we set
$ d_{\sigma}(f):=d^{\mathsf p}(x, \sigma)$ for any point  $x$ in $f$.

For a $3\,2\,1$-avoiding permutation $\sigma$, we define  the path depth difference $\delta_\sigma(v)$
of a boundary vertex $v$  to be
\begin{equation}\label{eq:ell}
\delta_\sigma(v)= d_{\sigma}(f_{v}^r)-d_{\sigma}(f_{v}^l),
\end{equation}
where $f_{v}^l$  is the face to the left of $v$  and  $f_{v}^r$ is  the face to the right of   $v$.   In  \eqref{eq:depthconf} below, we will show that $\delta_\sigma(v)$ is always either $0$ and $1$ for every boundary vertex $v$.

\begin{example}\label{exam:perm}
For $\sigma=
\begin{pmatrix}
1&2&3&4&5&6&7&8&9&10&11&12 \\
4&1&6&2&3&7&8&9&12&5&10&11
\end{pmatrix}
$,
the path depths of the internal faces are displayed here.   Also when
 $\delta_\sigma(v) = 1$ (resp. $=0$)  for a boundary vertex $v$ we write $\delta_\sigma = 1$
 (resp. $\delta_\sigma = 0$).
  $$
\grim{0.8}{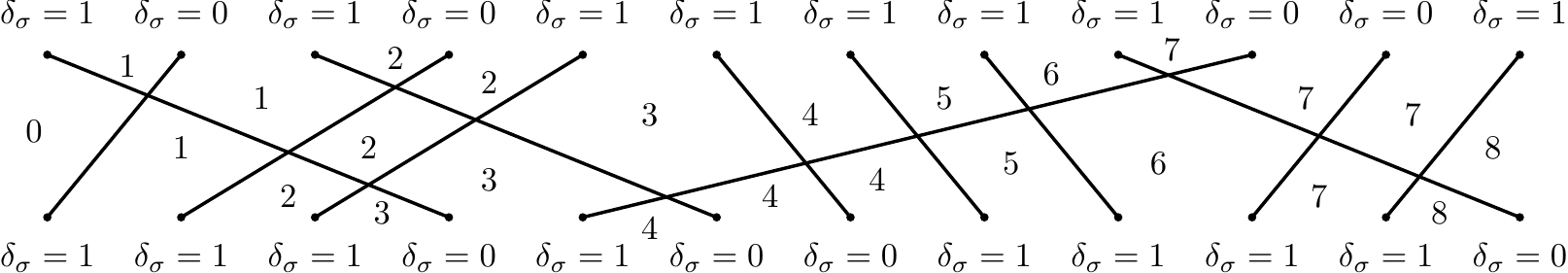}.
$$
\end{example}

To distinguish the top and bottom vertices, we  denote the $i$th  vertex on the top  (\emph{resp. bottom}) boundary  of the diagram of $\sigma$  by $t_i$ (\emph{resp. $b_i$}).
Using the values $\delta_\sigma(t_i)$  and $\delta_\sigma(b_i)$, we create  standard tableaux $\SF^+$ and  $\SF^-$ of at most two rows by assigning the position of $i$ as follows:
\begin{definition} \label{def:inverse}  For $i = 1,\dots, \ell$,
\begin{itemize}
\item put $i$ in the first row of $\SF^+$  (\emph{of $\SF^-$}) if $\delta_\sigma(t_i)=1$ (\emph{if $\delta_\sigma(b_i)=1$});
\item put $i$ in the second row of $\SF^+$  (\emph{of $\SF^-$}) if $\delta_\sigma(t_i)=0$  (\emph{if $\delta_\sigma(b_i)=0$});
\item let $\Psi(\sigma) = (\SF^+, \SF^-)$.
\end{itemize}
\end{definition}

Now we assume $\sigma$  corresponds  to a pair of standard tableaux $(\R,\Q)$ by the RS correspondence,  where $\R$ is the recording tableau and $\Q$ is the insertion tableau. Note that the permutation $\sigma$ in Example~\ref{exam:perm} corresponds to
$\left(\; \grim{0.8}{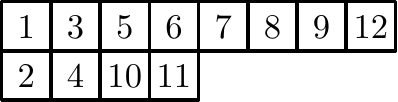}\; \; , \;\; \grim{0.8}{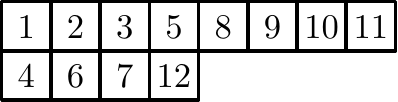} \; \right)$ by the RS correspondence.

The following is the main theorem of this section,  which we prove after a sequence of results.
\begin{thm}\label{thm:RSpairs} Let $\sigma$ be a $3\,2\,1$-avoiding permutation,  and assume  $\Psi(\sigma) = (\SF^+, \SF^-)$ as in Definition \ref {def:inverse}.
Suppose  $\sigma$ corresponds to $(\R,\Q)$ under the RS correspondence.  Then $\SF^+=\R$ and $\SF^-=\Q$.
\end{thm}

\begin{definition} \label{def:LRFsets}
Let $\mathcal{L}$, $\mathcal{F}$,$\mathcal{R}$  be the subsets of  all $i \in \{1, \ldots, \ell\}$ such that
$$
\mathcal{L}=\{i \mid \sigma(i) < i\},  \quad
\mathcal{F}=\{i \mid  \sigma(i) =  i\}, \quad
\mathcal{R}=\{i \mid \sigma(i) > i\}.
$$
\end{definition}

  \begin{lemma}\label{lem:s5rs} \hfil  \begin{itemize}
\item[{\rm (i)}]  If  $r,s\in \mathcal{F}\cup \mathcal{R}$ (or  if $r,s \in  \mathcal{F} \cup \mathcal L$)  with $r<s$, then $\sigma(r) < \sigma(s)$.
\item[{\rm (ii)}]  If $s \in \mathcal{F} \cup \mathcal{R}$, then $s$ is in the first row of the recording tableau $\R$ of $\sigma$.
\end{itemize}
 \end{lemma} \begin{proof} Part (i) is a  consequence  of Lemma~\ref{lem:s5l1}.
For (ii),  let $s \in \mathcal{F} \cup \mathcal{R}$ and $i <s$.
If $i\in  \mathcal{F} \cup \mathcal{R}$, then we have $\sigma(i) < \sigma(s)$  by (i).
If $i \in \mathcal{L}$, then $\sigma(i) <i <s<\sigma(s)$. Hence for any $i<s$, we have $\sigma(i)<\sigma(s)$.   So in the Robinson-Schensted procedure,  we adjoin  $\sigma(s)$ to the end of the  first row of the insertion tableau, and therefore, $s$ is in the first row of recording tableau.
\end{proof}

For an  intersection of lines in the diagram  of $\sigma$, the  path depths of the faces around the crossing  must be
as in one of the following configurations:
\begin{equation}\label{eq:depthconf}
\grim{0.8}{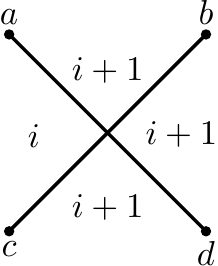} \; , \qquad
\grim{0.8}{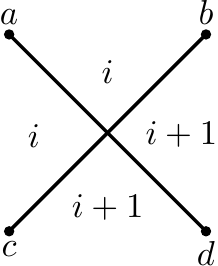} \; , \qquad
\grim{0.8}{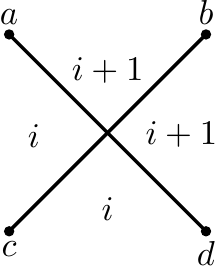} \; .
\end{equation}

Note that the  vertices $a,b,c,d$ in these pictures represent  either a vertex on a boundary or a crossing point of two vertical lines.  If  $a$ is a vertex on the top boundary, then the second configuration is not possible,  since the shortest path from a point in $f_a^r$ to $f_\infty$ must cross an edge to get to $f_a^l$ as shown below.  Likewise, if $c$ is a bottom
vertex, then the third configuration is impossible.
$$
\grim{0.8}{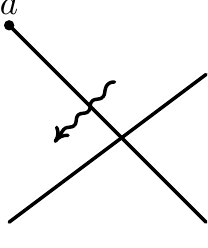} \; ,  \qquad \qquad  \grim{0.8}{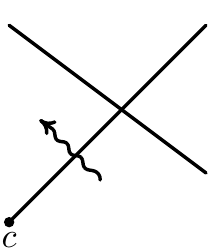}\; .
$$
\begin{lemma} \label{lem:ellcond1}  \begin{itemize}
\item[{\rm (i)}]  If $s \in \mathcal{F}  \cup \mathcal{R}$, then  $\delta_\sigma(t_s)=1$.
\item[{\rm (ii)}] If  $s \in \mathcal{L}$, then   $\delta_\sigma(b_{\sigma(s)})=1$.
\item[{\rm (iii)}]  If  $s \in \mathcal{F}$, then $\delta_\sigma(b_{\sigma(s)})= \delta_\sigma(b_s) = 1$.
\end{itemize}
\end{lemma}
\begin{proof} Only the first and third configurations in \eqref{eq:depthconf} can occur when $a = s$ on the top boundary of the diagram of $\sigma$, so (i) is apparent.   Similarly, if $c = \sigma(s)$ for some $s \in \mathcal L$,
then only the first two configurations are possible and we have (ii).  Part (iii) is clear.
\end{proof}

We have seen in Lemma \ref{lem:s5rs} that each $s \in \mathcal F \cup \mathcal R$ lies in the first row of the recording tableau $\R$.  To determine conditions for  $s \in \mathcal{L}$ to lie in the first row of  $\R$ will require some additional concepts.

Assume  $s\in \mathcal{L}$ and start moving along the edge from vertex $t_s$ on top. Whenever we reach a crossing, we make one of the following allowable turns:
\begin{equation}\label{eq:alloedturns}
\grim{0.8}{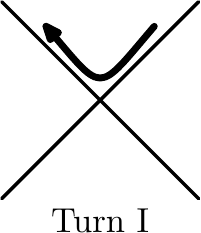} \;  \qquad  \qquad
\grim{0.8}{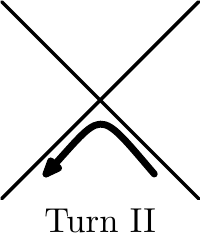} \;.
\end{equation}
We keep traveling  until we arrive at a vertex
on the top or bottom boundary and then stop.
The following  illustrates what happens in  Example~\ref{exam:perm}.
\begin{example}
$$
\grim{0.75}{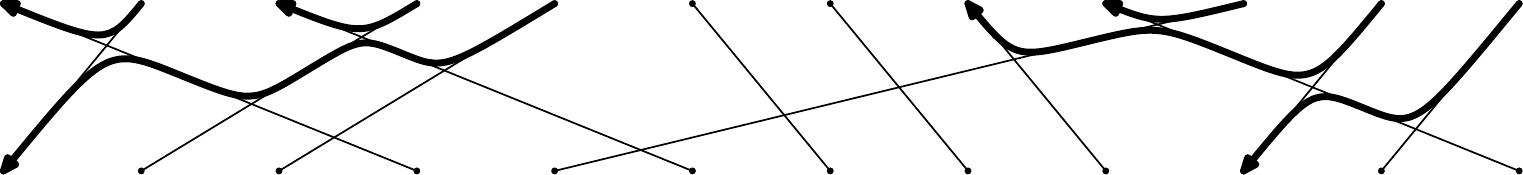} \;.
$$
\end{example}
Let $\varepsilon(t_s)$ be the number of turns we made to arrive at a vertex  in top or bottom starting from $t_s$.
The following lemma is  clear.
\begin{lemma}
For $s \in \mathcal{L}$,  starting from vertex $t_s$, we arrive at a vertex $t_r$ in the  top if and only if $\varepsilon(t_s)$ is odd, and in this case,  $r\in \mathcal{R}$.  We arrive at a vertex $b_r$  in the bottom if and only if $\varepsilon(t_s)$ is even, and in this case, $\sigma^{-1}(r) \in \mathcal{L}$.
\end{lemma}

As we traverse an edge,  there  are faces $f_1$ and $f_2$  on the two sides of that edge.
The absolute value of the difference  $|d_\sigma(f_1)-d_\sigma(f_2)|$ is either $0$ or $1$  by (\ref{eq:depthconf}).
When the edge is connected to a vertex $v$ on either  the top or bottom boundary, the path depth difference equals $\delta_\sigma(v)$.
The differences switch back and forth between $0$ and $1$ as we make a turn (see \eqref{eq:depthconf} and \eqref{eq:alloedturns}).

Therefore, starting from vertex $t_s$, $s\in \mathcal{L}$,  if we make an odd number of turns and  arrive at vertex $t_r$, $r\in \mathcal{R}$, then $\delta_\sigma(t_s)$ and $\delta_\sigma(t_r)$ are different. If instead we make an even number of turns and arrive at  vertex $b_r$, then $\sigma^{-1}(r)\in \mathcal{L}$, and $\delta_\sigma(t_s) = \delta_\sigma(b_r)$.

Now from Lemma~\ref{lem:ellcond1}, we have the following corollary.
\begin{cor}\label{C:delsig}
For $s \in \mathcal{L}$, when we start from the vertex $t_s$, we end at a vertex on the  top if and only if  $\delta_\sigma(t_s)=0$.
\end{cor}

Let $s\in \mathcal{L}$. Assume $s$ is in the second row of the recording tableau $\R$.   Hence,  $\sigma(r)$   is bumped out of the first row of $\Q$  by $\sigma(s)  <\sigma(r)$ for some $r<s$ while applying the Robinson-Schensted process. And by Lemma~\ref{lem:s5l1},  $r\in \mathcal{R}$.
 Note any $\sigma(r)$ can be bumped out  at most once.
Thus,  each $s\in \mathcal{L}$ in the second row of $\R$ is paired with a unique $r\in \mathcal{R}$ with $r<s$ such that  $\sigma(r) >\sigma(s)$.  We call such a pair $(r,s)$  an  \emph{RS pair}.

\begin{example}
 For
$$
\sigma=
\begin{pmatrix}
1&2&3&4&5&6&7&8&9&10&11& 12 \\
4&1&6&2&3&7&8&9&12&5&10&11
\end{pmatrix},
$$
the crossings in the diagram of $\sigma$ made by the  RS pairs are indicated by squares in the following example.
$$
\grim{0.75}{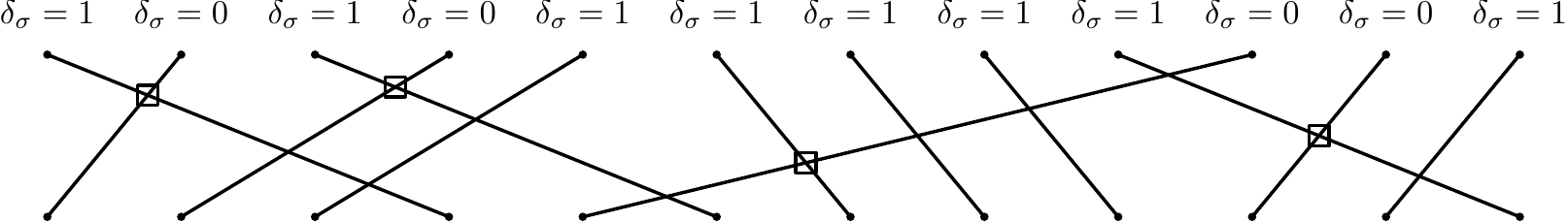} \;.
$$
Recall $\sigma$  corresponds  to
$(\R,\Q) = \left(\; \grim{0.8}{tabR.pdf}\; \; , \;\; \grim{0.8}{tabQ.pdf} \; \right)$.
\end{example}

Assume $s\in \mathcal{L}$, and there is an RS pair $(r,s)$. We will show that a  Turn II is always followed by a Turn I during the trip starting from $t_s$, and hence $\delta_\sigma(t_s) = 0$ by Corollary \ref{C:delsig}. The following illustrates the possibilities that occur when a Turn II occurs.
$$
\grim{1}{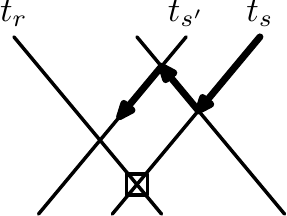}\; , \qquad
\grim{1}{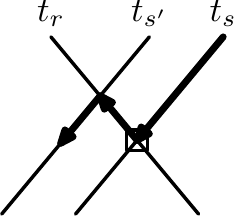} \;.
$$
In either case, since $r<s'<s$ and $s'$ is not paired with $r$, there must be an $r'$  such that $r'<r$,\,  $r'\in \mathcal{R}$, \, $(r',s')$ is an RS pair, and a Turn I follows:
$$
\grim{1}{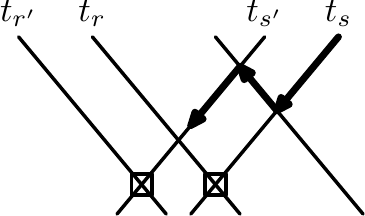}\; , \qquad
\grim{1}{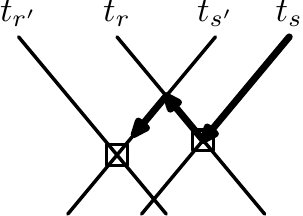}\;.
$$
We repeat this procedure until we arrive at a vertex on the top boundary in this case.

On the other hand, assume that $s\in \mathcal{L}$ and $s$ is in the first row of the recording tableau $\R$, i.e., there is no RS pair of $s$.
We show that if a Turn I happens,  then always a Turn II follows during the trip starting from $t_s$, and hence $\delta_\sigma(t_s) = 1$ by Corollary \ref{C:delsig}.
In order for a turn to occur,  the line  traveled is crossed by a line starting from a vertex $t_r$ with $r\in \mathcal{R}$, $r<s$, and $\sigma(r) > \sigma(s)$. Since $s$ is not paired, $r$ must be paired with another $s' \in \mathcal{L}$ with  $r<s'<s$.
If there is no $s''\in \mathcal{L}$ such that $s'<s''<s$, then we make a Turn II  when  the line from $t_{s'}$ crosses the line from $t_r$.
$$
\grim{1}{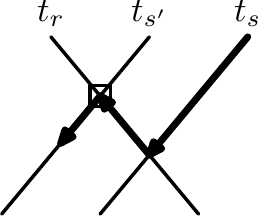}
$$
If there is  $s''\in \mathcal{L}$ such that $s'<s''<s$, then we make a Turn II at the crossing of the line from $t_{s''}$ with the line from $t_r$. If $s''$ is paired with a $r'\in \mathcal{R}$,  then $r'>r$ since $s'<s''$ is paired with $r$.
$$
\grim{1}{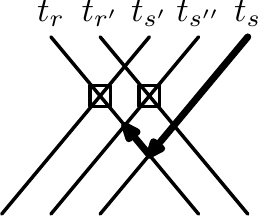}
$$
Then, the line from $t_{r'}$ must meet with the line from $t_{s}$,  and we get a contradiction,  since we made a turn at the crossing of the line from $t_{r}$ and the line from $t_{s}$. Hence, $s''$ is not part of an RS pair,  and we make a turn at the crossing of the line from $t_{r}$ and the line from $t_{s''}$.
$$
\grim{1}{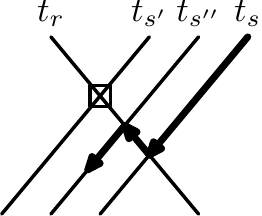}
$$
We continue this procedure until we arrive at a vertex on the bottom boundary in this case.

We now have  the following result,  which summarizes the relationship we have
determined  between  the path depth differences and the positions in the recording tableau $\R$ of $\sigma$.

\begin{lemma}\label{lem:c2c2} Assume $\R$ is the recording tableau of $\sigma$ and $\mathcal{L}$
is as  in Definition \ref{def:LRFsets}.
\begin{enumerate}[{\rm (a)}]
\item  If $s \in {\mathcal R}\cup {\mathcal F}$, then  $\delta_\sigma(t_s)=1$ and $s$ is in the first row of  $\R$.
\item  If $s \in \mathcal{L}$ and   $\delta_\sigma(t_s)=1$,  then  $s$ is in the first row of  $\R$.
\item  If $s \in \mathcal{L}$ and   $\delta_\sigma(t_s)=0$, then $s$ is in the second row of  $\R$.
\end{enumerate}
\end{lemma}
Now we are ready to prove the main theorem in this section.

\begin{proof}(\textbf{Proof of Theorem \ref{thm:RSpairs}}).
From a $3\,2\,1$-avoiding permutation  $\sigma$, we
obtained a pair  of standard tableaux $\Psi(\sigma) = (\SF^+,\SF^-)$ using the path depth differences in  Definition~\ref{def:inverse} and a pair of standard tableaux  $(\R,\Q)$ by the RS correspondence.   Lemma~\ref{lem:c2c2}  shows that  $\SF^+ = \R$.  Since $\Psi(\sigma^{-1}) = (\SF^-, \SF^+)$ and  $\sigma^{-1}$ corresponds to $(\Q,\R)$,
we have   $\SF^-=\Q$.
\end{proof}

\begin{remark}\label{rem:tlbijection}
Starting a trip from $t_s$, $s\in  \mathcal{L}$, we only need Turn I and Turn II.
However, beginning  the trip from any vertex on the top (or bottom),  also requires  using the following turns.
$$
\grim{0.8}{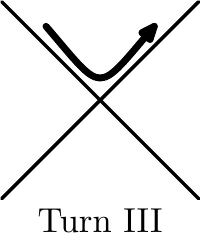} \; , \qquad
\grim{0.8}{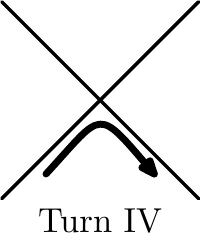} \; .
$$
Connecting the initial and  terminal vertices of each path traveled results in a planar diagram
(no edges cross)  with each of the $2\ell$ vertices connected to a unique vertex.    This determines a bijection between the set of $3\,2\,1$-avoiding permutations and Temperley-Lieb diagrams.  In particular,
Example~\ref{exam:perm} gives
$$
\grim{0.75}{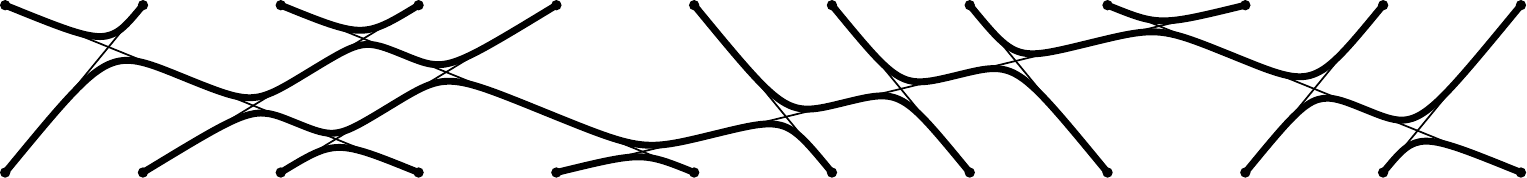} \; ;
$$
hence,  the following Temperley-Lieb diagram
$$
\grim{0.75}{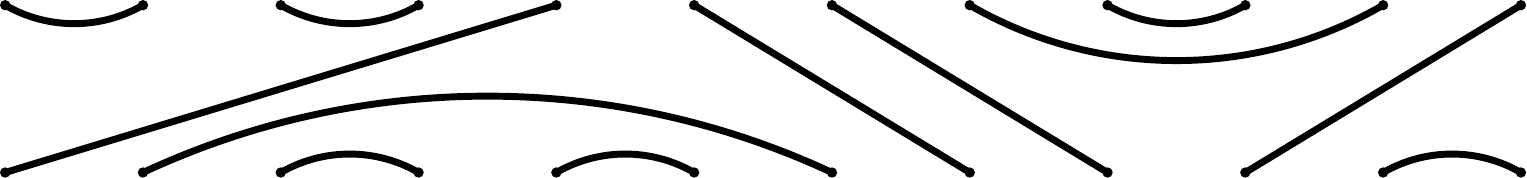}  \; .
$$

\end{remark}

\section{$\Phi$ is bijective} \label{sec:inverse}

In Section~\ref{sec:bijection}, we constructed a  map $\Phi$   from
the set of  semistandard tableaux of shape $(3^n)$ and type  $\{1^2,\ldots, k^2, k+1, \ldots, 3n-k\}$ to the set of nonelliptic $\Atwo$-webs with a boundary of $k$ pluses  on the top and a boundary of  $3n-2k$  minuses  on the bottom. In this section, we  show that $\Phi$ is a bijection by  using  Khovanov-Kuperberg's  definition (Definition~\ref{def:circle}) of path depth.

We connect the rightmost vertex on the top boundary  of an $\Atwo$-web
to the rightmost vertex on the bottom boundary to form one boundary line, and consider the unbounded region of  the web that opens to left  as  $f_\infty$.  Then, we assign a path depth  as in  Definition~\ref{def:circle} from each internal face of the web to $f_\infty$.
We will use this web path depth to construct two standard tableaux from the web,
providing the inverse of $\Phi$.  This is analogous to what we did in Section~\ref{sec:combinatorics},
where we used  the path depth of a $3\,2\,1$-avoiding permutation  to construct two standard tableaux
that were exactly the same tableaux produced from the  inverse of the RS correspondence by Theorem \ref{thm:RSpairs}.

Recall from \eqref{eq:ell}  that $\delta_\sigma(v)$ is the path depth difference for
a boundary vertex $v$ of  a $3\,2\,1$-avoiding permutation $\sigma$, and $\delta_\sigma(v) = 0$ or $1$
(see \eqref{eq:depthconf}).
In an analogous manner,  we define $\delta(v)$ for a boundary vertex of a  nonelliptic $\Atwo$-web.
The only difference is that $\delta(v)$ may have values $0,1$ and $-1$. (We will see below that $\delta(v)$  has  only these values.)
Using the values $\delta(t_i)$  and $\delta(b_i)$, we create two standard tableaux $\SF^+$ and  $\SF^-$ of at most three rows by inserting  $i$ as follows:
\begin{itemize}
\item put $i$ on the first row of $\SF^+$  (\emph{of $\SF^-$})
if $\delta(t_i)=1$ (\emph{resp. if $\delta(b_i)=1$});
\item put $i$ on the second row of $\SF^+$  (\emph{of $\SF^-$})
if $\delta(t_i)=0$ (\emph{resp. if $\delta(b_i)=0$});
\item put $i$ on the third row of $\SF^+$  (\emph{of $\SF^-$})
if $\delta(t_i)=-1$ (\emph{resp. if $\delta(b_i)=-1$}).
\end{itemize}
Note that this is an extension of Definition~\ref{def:inverse}.
From \cite{K} and \cite{PPR}, we know that the set of  semistandard
tableaux of shape $(3^n)$ and type  $\{1^2,\ldots, k^2, k+1, \ldots,  3n-k\}$
and set of nonelliptic $\Atwo$-webs with $k$ pluses  on the top
boundary and   $3n-2k$ minuses on the bottom boundary have  the same cardinalities.
Hence, to show that $\Phi$ is bijective, it is enough to prove
that $\Phi$ is injective,  which amounts to showing that we can recover $\T$ from $\Phi(\T)$.
Observe that if we can recover $\T^+$ and $\T^-$, then getting $\T$ is straightforward by applying $\varphi^{-1}$
(for $\varphi$ as given in Theorem \ref{thm:cute}),   adjusting
entries, rotating tableaux and joining them together. So, we will only show that we can  determine $\T^+$ and $\T^-$ from $\Phi(\T)$ using the path  depth differences  $\delta(v)$.

We start by noting that path depth  differences  $\delta(v)$ of boundary vertices $v$  on $\Phi(\T)$ are same as those on $\Phi_*(\T)$ given in Step 10 of Section.
The diagram $\Phi_*(\T)$ is  formed by  the diagrams of the two webs  $\W^+$  and $\W^-$ and the diagram of a $3\,2\,1$-avoiding permutation $\sigma$ on the isolated vertices.

Let $v$ be an isolated vertex on the top boundary of $\Phi_*(\T)$ and consider its path depth difference $\delta_{\sigma}(v)$   as a vertex in  the diagram of the  permutation
$\sigma$ as in Section \ref{sec:combinatorics}.  (An isolated vertex on the bottom boundary can be treated by a similar argument.)   Corollary~\ref{lemma:LR} implies that only the following two configurations can occur when an  $\EM$-configuration
is added to the diagram of $\sigma$.   The value of the  difference does not
change by the addition of the $\EM$-configuration; hence $\delta(v) = \delta_\sigma(v)$ for isolated vertices.  This implies that  $\mathsf{T}^+_\diamond$ and $\mathsf{T}^-_\diamond$ can be recovered  by Theorem~\ref{thm:RSpairs}  from the path depth
$\delta(v)$  of the isolated vertices:
$$
 \grim{0.6}{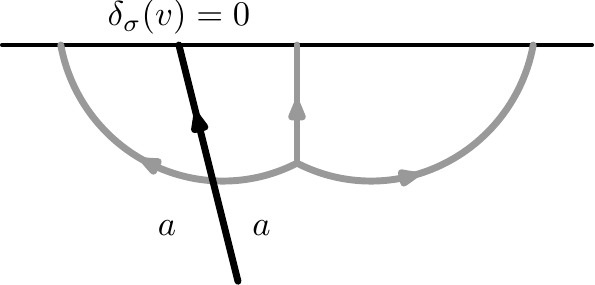} \longrightarrow \grim{0.6}{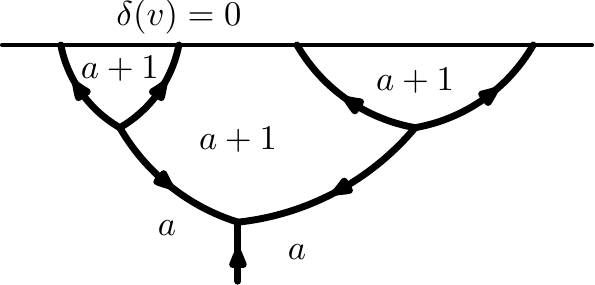}
$$
$$
\grim{0.6}{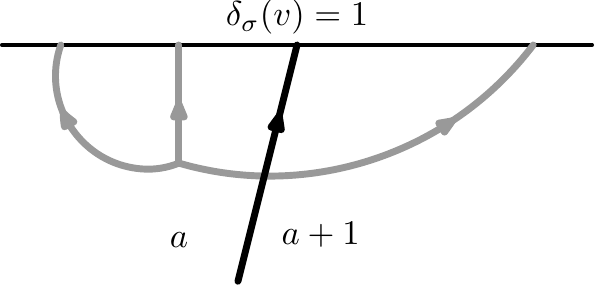} \longrightarrow \grim{0.6}{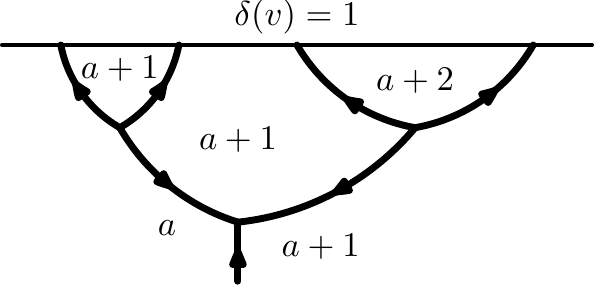}
$$
(Path  depths  are indicated as $a$, $a+1$, or $a+2$.)

Now assume that we have an $\EM$-configuration in $\W^+$ (or $\W^-$).  Adding a vertical line connecting an isolated vertex  on the top boundary with an isolated vertex on bottom boundary to the $\EM$-configuration results in no change to $\delta(v)$ for a vertex $v$ in the $\EM$-configuration,  since the shortest path from $f_v^r$ and $f_v^l$ to $f_\infty$ either will cross the vertical line if $v$ is to the right of the line or will not cross the vertical line if $v$ is to the left of the line.  Hence, the path depth difference of a vertex in $\W^+$ (resp. $\W^-$)
is the same as it is in $\Phi_*(\T)$.   Thus, by Tymoczko's result (see Proposition~\ref{depth}), the path  depths in  $\W^+$
and the circle depths in $\M^+$ are the same (resp. of $\W^-$ and of $\M^-$ are the same).   Then we can recover
  $ \T^+_\Box$ from $\M^+$  ($\T^-_\Box$ from $\M^-$)  by  $\delta(t_i)=1,0,-1$ (or $\delta(b_i)=1,0,-1$)  if and only if $i$ in the first, second or third row of $\T^+_\Box$ (or $\T^-_\Box$), respectively.

Once  $\T^+_\Box$ and $\T^+_\diamond$ are determined, then $\T^+$ can be obtained by putting $i$ in row $r$  of $\T^+$ if and only if $i$ is in row $r$  of $\T^+_\Box$ or $\T^+_\diamond$. Similarly, we can recover $\T^-$ from $\T^-_\Box$ and $\T^-_\diamond$.

In summary, we uniquely determined $\T$ from $\Phi(\T)$ using $\delta(v)$, and hence
$\Phi$ is injective (and bijective, too).

The following is our example from Section~\ref{sec:bijection}. The path depths are indicated on $\Phi_*(\T)$ and $\Phi(\T)$.

$$
\grim{0.45}{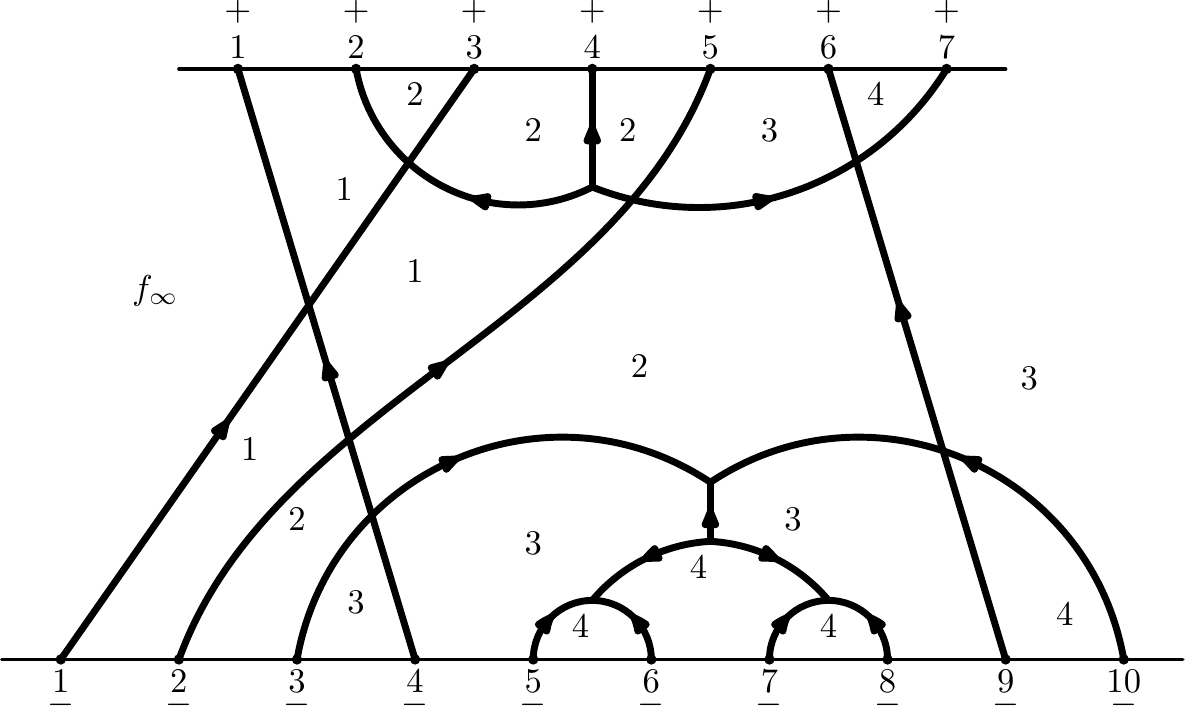} \quad
\grim{0.45}{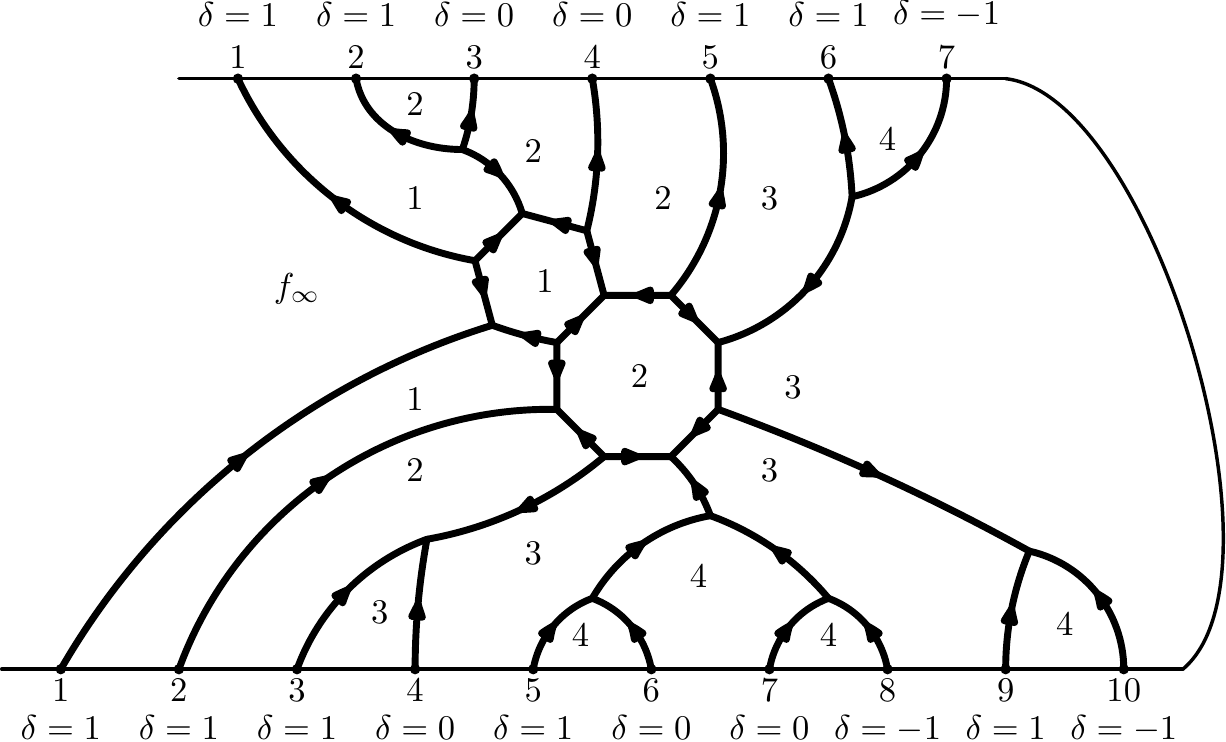}
$$
Then, using the path depths, we recover the standard tableaux
$$
\T^+=\grim{0.8}{exam4.pdf} \   \text{ and } \ \T^-=\grim{0.8}{exam8.pdf}\;.  \hspace{.8 truein}   \square
$$

\section{Connections with Westbury's flow diagrams} \label{sec:flow}
The algorithm in Section \ref{sec:bijection} produces an $\Atwo$-web from
a pair $(\T^+, \T^-)$  of standard tableaux obtained from a semistandard tableau $\T$.
Westbury  \cite{We}  has given an alternative to $\Ann$-webs connected with $\mathfrak{sl}_r$ using flow
diagrams in triangular shapes.
In this final section, we describe  how to  obtain an $\Atwo$ flow diagram  from such  a  pair  $(\T^+, \T^-)$ of standard tableaux.   The results described here for $\Atwo$ can be extended to $\Ann$ for arbitrary $r$,
and more details on this will appear elsewhere. Our aim here is to describe the relationship between  the results of
the previous sections to the flow diagram approach.

First,  we briefly review Westbury's result.
To  $a\in\{1, \ldots, r\}$, we associate two equilateral triangles whose sides are one unit in length:
\begin{equation}\label{eq:triangle1}
\grim{0.7}{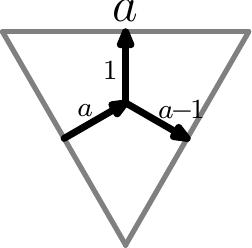}, \qquad \grim{0.7}{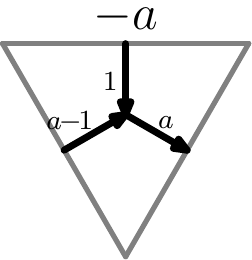} \; .
\end{equation}
For compatibility with the  conventions in previous sections, our arrows are the reverse of
what is in \cite{We},  so  they go from  $-$ to $+$
rather than from $+$ to $-$.
The edge labels  $0$ and $r$ are omitted in the diagrams,  and  the triangles associated to $1,-1, r ,-r$ are
$$
\grim{0.7}{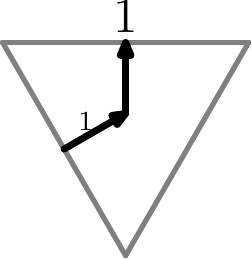}, \quad
\grim{0.7}{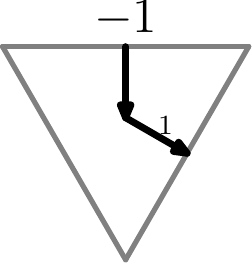}, \quad
\grim{0.7}{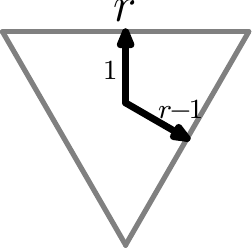}, \quad
\grim{0.7}{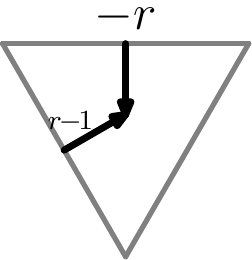} \; .
$$

Now to each pair $(s,t)$, $s,t\in \{0,1, \ldots, r-1\}$, we associate a rhombus such that when $st \neq 0$,
\begin{equation}\label{eq:rhombus}
\grim{0.7}{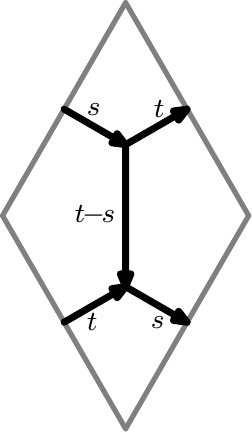}\quad \text{if $s\ne t$,} \qquad  \quad
\grim{0.7}{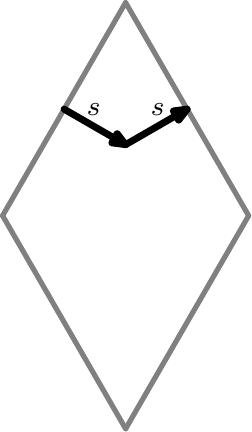}\quad \text{if $s= t$.}
\end{equation}
If $t=0$ or $s=0$, the corresponding rhombi are
$$
\grim{0.7}{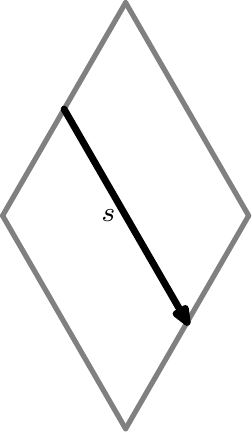}\;  \qquad \text{or} \qquad \grim{0.7}{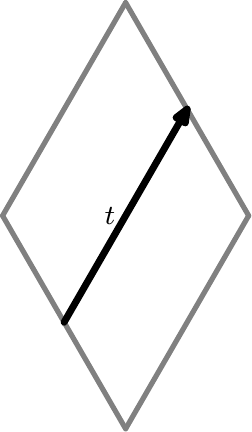} \;.
$$

Reversing the orientation of an edge changes the sign of the label:
$$
\grim{0.7}{rhombst.pdf}\ = \ \grim{0.7}{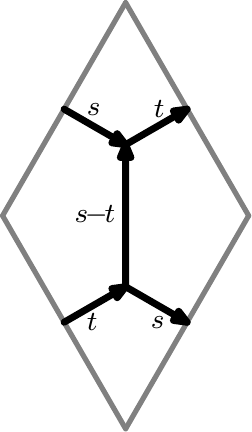}.
$$
Also for $\Ann$,  an edge labeled with $-q$, \ $(1\le q \le r-1)$ \   is identified with an edge labeled with  $r-q$.
For example, in the $\Atwo$ case we have
$$
\grim{0.7}{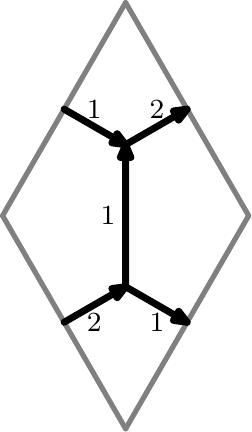}=\grim{0.7}{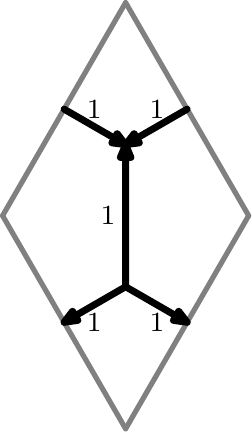}.
$$

We associate to a word $w$ of  $n$ symbols from $\{\pm 1,\ldots, \pm r\}$  an equilateral triangle of length $n$. Along the top boundary of this triangle, we place triangles of unit length as in \eqref{eq:triangle1},  so that the numbers on the top of these triangles, when read from left to right,  give the word $w$.   Then, we fill in the rest of the triangle with rhombi as described above,  so that the labels on the edges  match  on the boundaries of the rhombi. The triangular diagram obtained in this way is the  \emph{flow diagram} determined by $w$ (see \cite{We} for more details).
The following is an example of $\Atwo$ flow diagram associated to the word  $w = \bar 1\,2\,2\,\bar 3\,1$. (Note we use
$\bar a$ instead of $-a$ when writing a word to simplify the notation.)
$$
\grim{0.7}{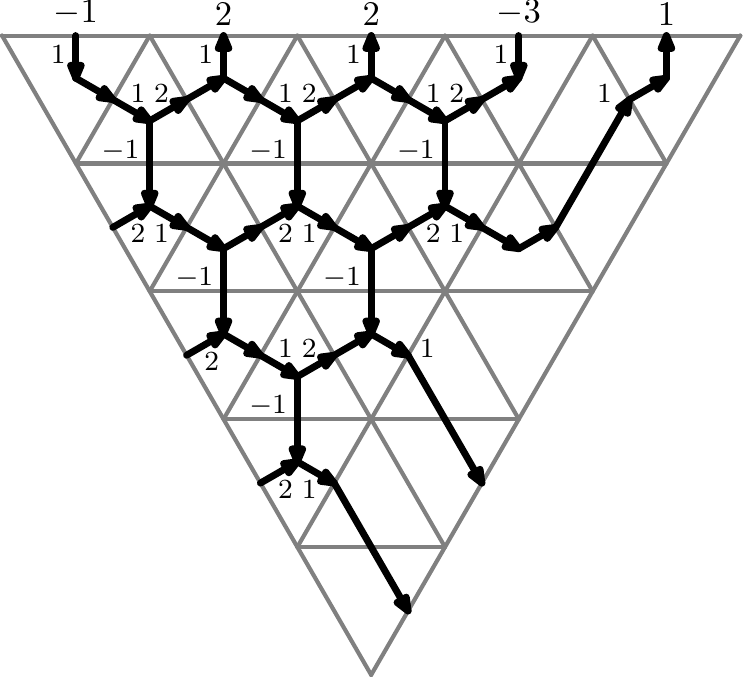} .
$$

We say a word $w$ (or its associated flow diagram) is \emph{admissible} if the associated flow diagram of $w$ has no contact with the boundaries of the large triangle  other than with the top boundary of triangle. (The diagram above is not admissible.)

Let  $\VV_+ = \VV$ be the natural $r$-dimensional $\mathfrak{sl}_r$-module and $\VV_- = \VV^*$  be  the dual module.   A tensor product of  copies of $\VV_+$ and $\VV_-$  corresponds to a string
$\smf = \smf_1\cdots \smf_n$   of symbols $+$ and $-$ and is denoted  $\VV_{\smf} = \VV_{\smf_1}\ot \cdots \ot\VV_{\smf_n}$.
Each  word $w=w_1\cdots w_n$ in $\{\pm 1,\ldots, \pm r\}$ of length $n$ determines a sign string $\smf(w)=\smf_1\cdots \smf_n$  given by $\smf_i=+$ if $w_i>0$ and $\smf_i=-$ if $w_i<0$.
For a given sign string $\smf$ of length $n$, we let $\mathcal{B}(\smf)$ be the set of admissible words $w$ in $\{\pm 1,\ldots, \pm r\}$ of length $n$ whose sign string $\smf(w) =\smf$.  Then the following analog of Theorem \ref{T:inv} holds in this setting:
\begin{thm}[{\cite[Cor.~4.6]{We}}]
Let $\smf$ be a sign string and
$\mathcal{B}(\smf)$ be the set of admissible words in $\{\pm 1,\ldots, \pm r\}$ whose sign string is  $\smf$.
Then, there is a bijection between $\mathcal{B}(\smf)$ and a basis for the space $\mathsf{Inv}(\VV_\smf)$ of  $\fsl_r$-invariants
 in $\VV_\smf$.
\end{thm}

Any interior face of a flow diagram is bounded by at least six edges,  since the following rhombus configuration is not allowable:
$$
\grim{0.7}{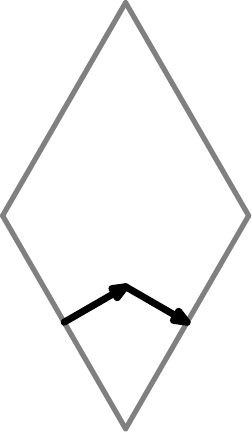}.
$$
Therefore, the nonelliptic $\Atwo$-webs discussed earlier in this paper are indeed $\Atwo$ flow diagrams.
Thus, in the  $\Atwo$ case, what we want is a bijection (in fact an injection) mapping the pair $(\T^+, \T^-)$ in Section~\ref{sec:bijection} to an admissible flow diagram.  To achieve this,  we apply the following algorithm,  which we
illustrate with the  running example from Section \ref{sec:bijection},  where
$\T^+=\grim{0.8}{exam4.pdf}$ has shape $\lambda = (4,2,1)$  and  $\T^-=\grim{0.8}{exam8.pdf}$ has shape $\mu=(5,3,2)$.
\bigskip

\noindent $\bullet$  \textbf{Break $\mathsf{T}^+$ into one-column tableaux}:
\medskip

\noindent Let  $\T(a,b)$ be the entry  in position $(a,b)$  of $\mathsf{T}^+$.
\begin{enumerate}[{(I)}]
\item  For $p=1, \ldots, \lambda_3$,   let $\T(2,j_p)$ be  the largest entry in the second row  of $\T^+$ less than $\T(3,p)$, such that  $\T(2,j_p)$ is different from  $\T(2,j_1)$, $\ldots$, $\T(2,j_{p-1})$ when $p > 1$.  Associate $\T(2,j_p)$ with $\T(3,p)$.
\item  For $p=1, \ldots, \lambda_2$,   let $\T(1,i_p)$ be  the largest entry in the first  row  of $\T^+$ less than $\T(2,p)$ such that  $\T(1,i_p)$ is different from  $\T(1,i_1)$, $\ldots$,  $\T(1,i_{p-1})$ when $p > 1$.  We associate $\T(1,i_p)$ with $\T(2,p)$.
\end{enumerate}
Now assemble associated entries of $\T^+$ to form one-column tableaux.  There will be  $\lambda_3$ tableaux of shape $(1^3)$, \ $\lambda_2-\lambda_3$ tableaux of shape $(1^2)$,  and  $\lambda_1-\lambda_2$ tableaux of shape $(1)$. For $\T^+$ as above,
we obtain
$\T^+_{3,1}=\grim{0.8}{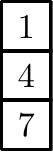}$,
$\T^+_{2,1}=\grim{0.8}{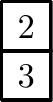}$,
$\T^+_{1,1}=\grim{0.8}{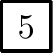}$,
$\T^+_{1,2}=\grim{0.8}{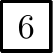}$.
The index $\ell$ in  $\T^+_{\ell, i}$ indicates that
the shape of the tableau is $(1^\ell)$.  To determine the second index,  we compare
the bottom entries in the tableaux of shape $(1^\ell)$  and number the tableaux so the
one with the smallest last row entry is $\T_{\ell, 1}^+$, the next one is $\T_{\ell,2}^+$, and so forth
until we reach the one with the largest last row entry.
\bigskip

\noindent $\bullet$ \  \textbf{Break $\mathsf{T}^-$ into one-column tableaux by the same procedure}:
\medskip

\noindent Thus, if  $\T^-=\grim{0.8}{exam8.pdf}$,  then
$\T^-_{3,1}=\grim{0.8}{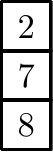}$,
$\T^-_{3,2}=\grim{0.8}{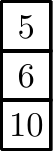}$,
$\T^-_{2,1}=\grim{0.8}{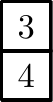}$,
$\T^-_{1,1}=\grim{0.8}{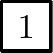}$,
$\T^-_{1,2}=\grim{0.8}{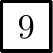}$.
\bigskip

\noindent $\bullet$ \textbf{Associate words in $\{\pm 1, \pm 2, \pm 3\}$ to the one-column tableaux}:
\medskip

Recall that $\T^+$ has $k$ entries and $\T^-$ has $3n-2k$ entries,  where $k = 7$ and $n=8$ in this particular
example.  From $\T^-_{3,1}=\grim{0.8}{tmin1.pdf}$, we get the  word $w=w_1\cdots w_{3n-k}$ of length $3n-k$ (which is 17 here)  where
$$
w_i=\begin{cases}
\bar 1 &\text{if $i$ is in the first row of $ \T^-$,} \\
\bar 2 &\text{if $i$ is in the second row of $ \T^-$,} \\
\bar 3 &\text{if $i$ is in the third row of $ \T^-$,} \\
\bullet &\text{otherwise, }
\end{cases}
$$
$$
 w=\bullet \bar{1} \bullet \bullet \bullet \bullet \bar{2}\bar{3}\bullet \bullet \,\big |\, \bullet \bullet \bullet \bullet \bullet \bullet \bullet  \; \null ,
$$
where  $\bar 1$ is in position 2, $\bar 2$ in position 7, and $\bar 3$ in position 8.
A wall  $\big |$ separates the $w_i$
from $\T^-$ on the left from  the $w_i$ on the right from $\T^+$.

The word gives the labels on the top boundary  of the triangle,
$$
\grim{0.6}{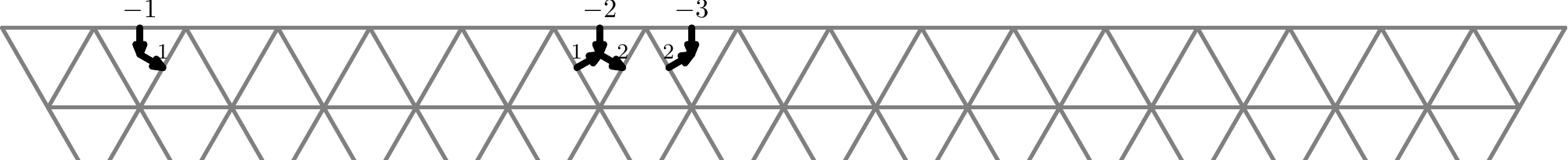} .
$$
\bigskip

\noindent $\bullet$ \  \textbf{Draw the flow diagram according to the boundary specifications:}

$$
\grim{0.6}{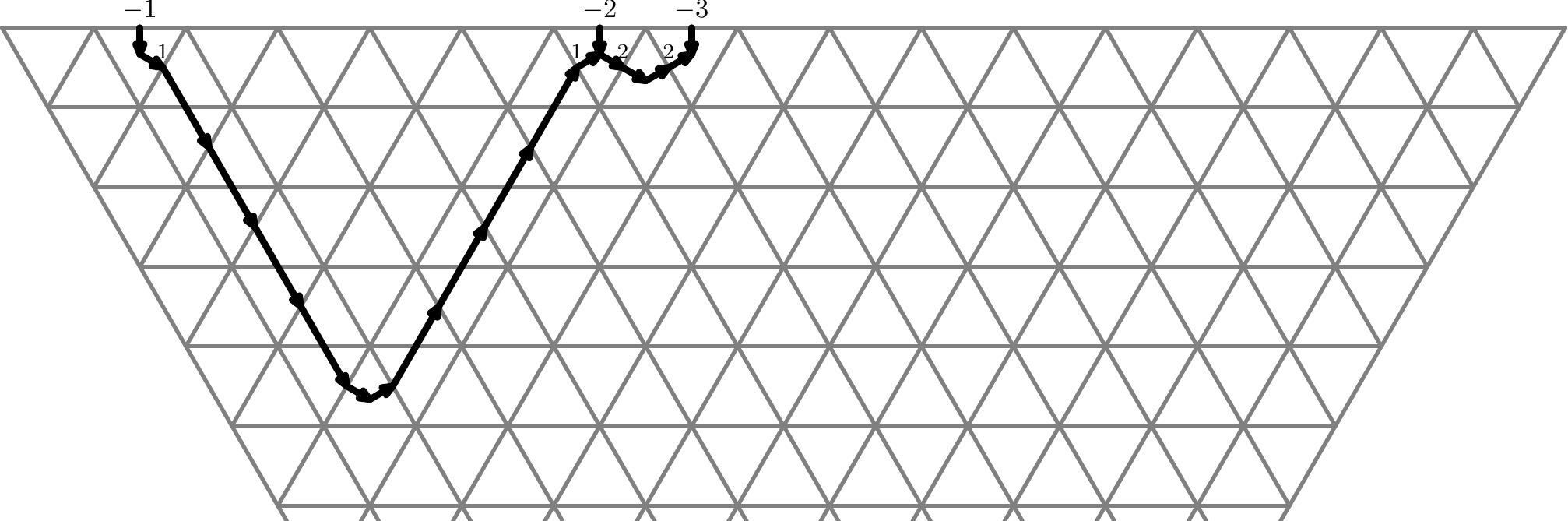} .
$$

\bigskip
\noindent $\bullet$ \  \textbf{Repeat the same procedure for the other one-column tableaux of length 3 coming from $\T^-$ and $\T^+$.}
\medskip
\begin{eqnarray*} && \T^-_{3,2}=\grim{0.8}{tmin2.pdf}:  \  \
\bullet \bullet \bullet  \bullet \,\bar{1}\bar{2} \bullet  \bullet  \bullet  \bar{3}\, \big | \bullet \bullet \bullet \bullet \bullet \bullet \bullet \; \null  \\
&& \T^+_{3,1}=\grim{0.8}{tplus1.pdf}: \ \
\bullet\bullet\bullet\bullet\bullet\bullet\bullet\bullet\bullet\bullet
\big | \, 3 \bullet \bullet  \, 2 \bullet \bullet \,1 \; ,
\end{eqnarray*}
where the word  $w=w_1\cdots w_{3n-k}$ from $\T^+_{3,1}$ is constructed from setting
$$
w_{17-i+1}=\begin{cases}
{1} &\text{if $i$ is in the first row of $ \T^+$,} \\
{2} &\text{if $i$ is in the second row of $ \T^+$,} \\
{3} &\text{if $i$ is in the third row of $ \T^+$,} \\
\bullet &\text{otherwise.}
\end{cases}
$$
The corresponding flow diagram is
$$
\grim{0.6}{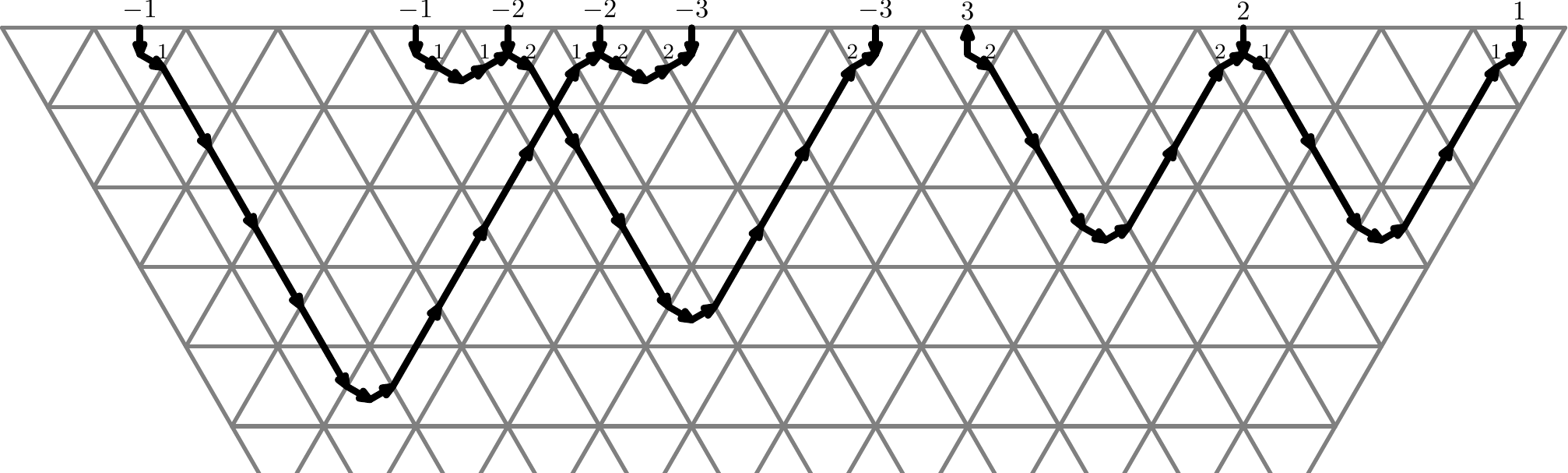} .
$$

Now for $\ell <3$, we pair $\T^-_{\ell,i}$  and $\T^+_{\ell,i}$ and to get an \emph{admissible}  flow diagram. The number
of such tableaux will always be the same for each $\ell=1,2$.  (Recall from Step 7 of Section \ref{sec:bijection} that $\T^+_\diamond$ and $\T^-_\diamond$ have the same shape.)
From  $\left( \T^-_{2,1}=\grim{0.8}{tmin3.pdf}\; , \;
\T^+_{2,1}=\grim{0.8}{tplus2.pdf} \right)$, we get the word
$$
w= \bullet \bullet \bar{1} \bar{2}\bullet\bullet\bullet\bullet\bullet\bullet \, \big | \, \bullet \bullet \bullet \bullet   2 1 \bullet \; ,
$$
and the flow diagram
$$
\grim{0.6}{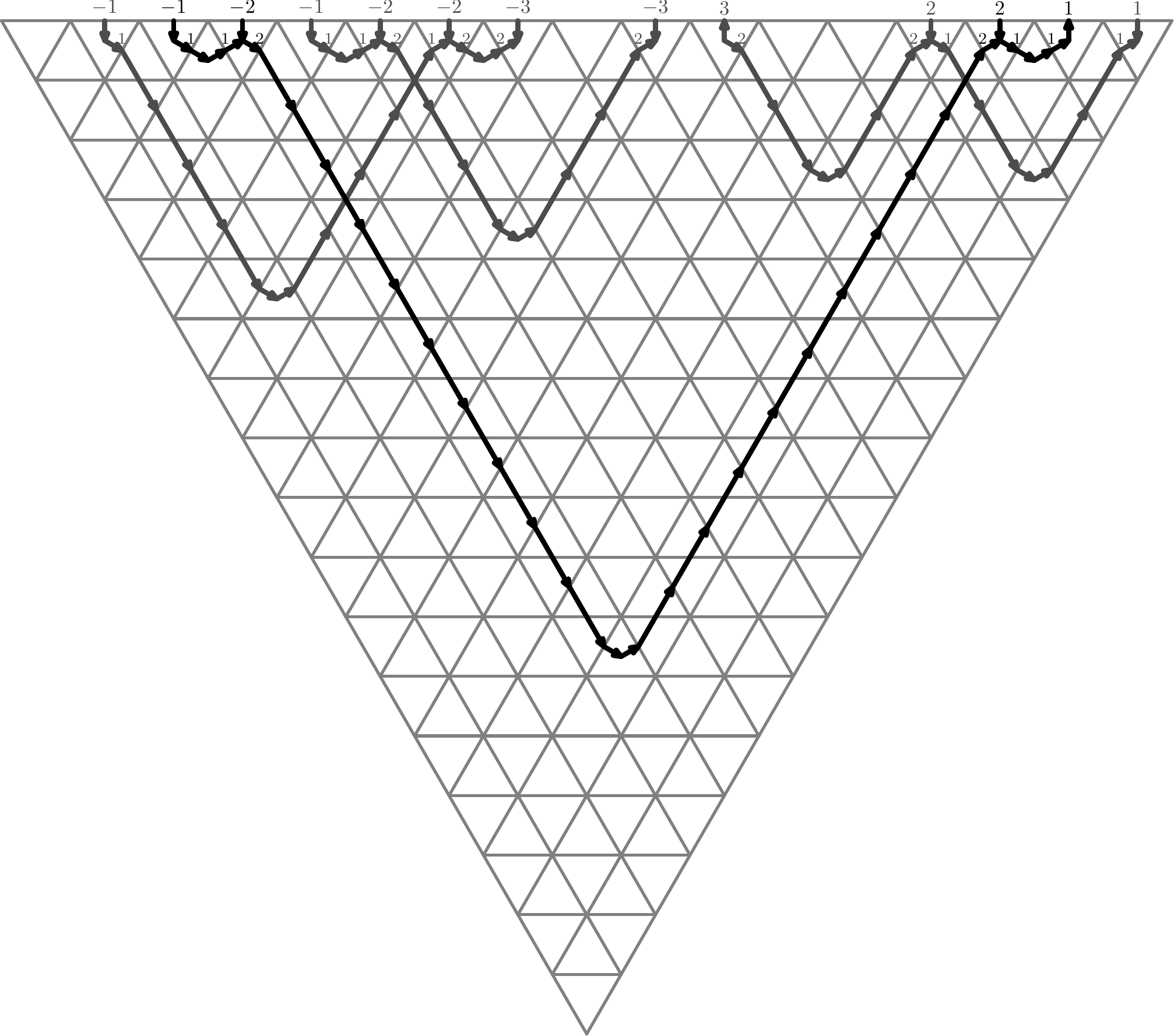} .
$$

Applying  a similar process to the pair  $\left( \T^-_{1,1}=\grim{0.8}{tmin4.pdf}\; , \;
\T^+_{1,1}=\grim{0.8}{tplus3.pdf} \right)$,   we get the word
$$
w= \bar{1} \bullet \bullet \bullet \bullet \bullet \bullet \bullet \bullet \bullet \,
\big | \, \bullet \bullet \bullet \bullet  \bullet 1 \bullet \bullet \bullet \bullet \null \; ,
$$
and to the pair  $\left( \T^-_{1,2}=\grim{0.8}{tmin5.pdf} \; , \;
\T^+_{1,2}=\grim{0.8}{tplus4.pdf} \right)$,   the word
$$
w= \bullet\bullet\bullet\bullet\bullet\bullet\bullet\bullet \bar{1} \bullet
\big | \bullet\bullet\bullet\bullet 1 \bullet\bullet\bullet\bullet\bullet
\null \; .
$$
The result of superimposing all the flow diagrams is the following flow diagram:
$$
\grim{0.6}{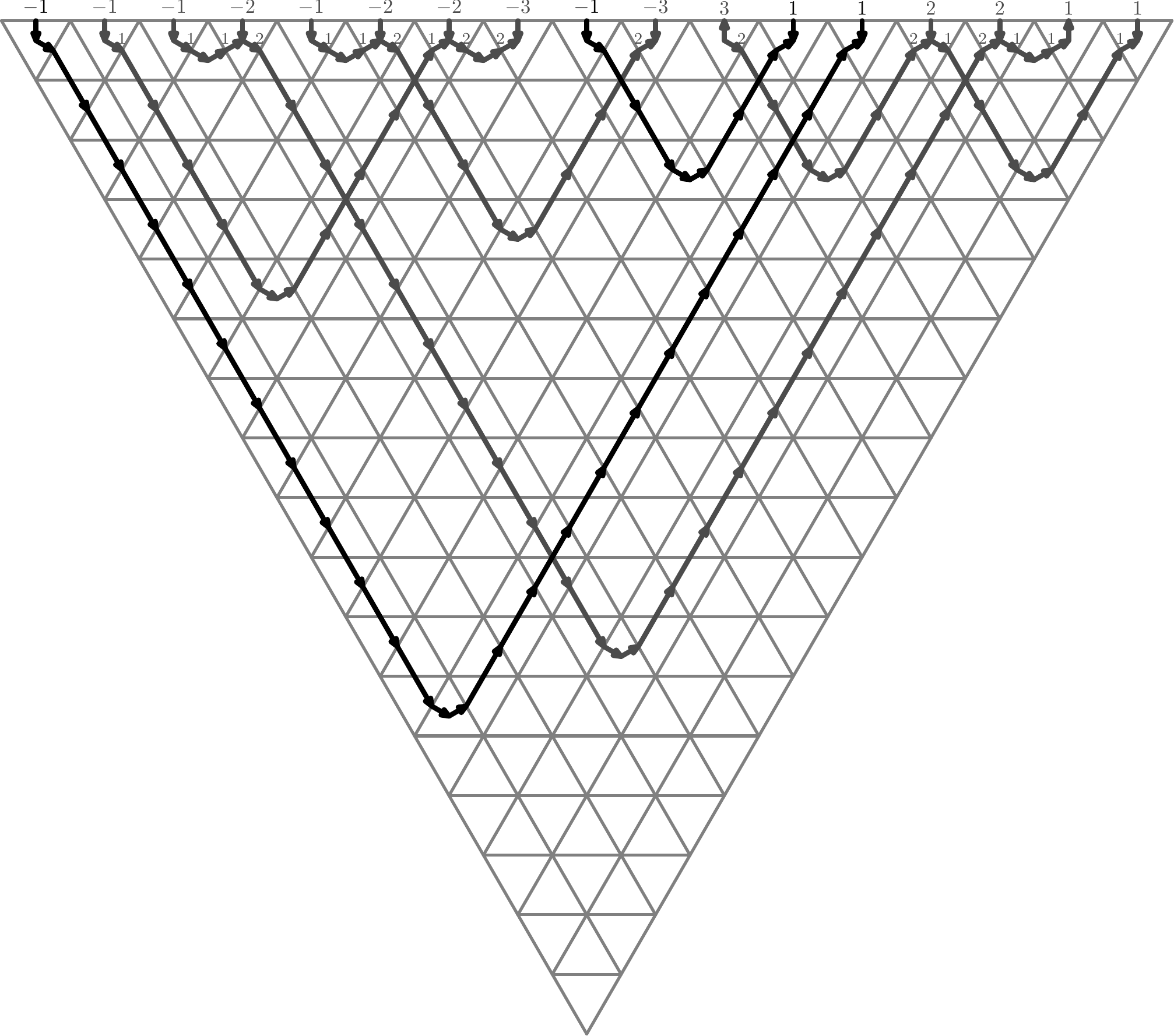} .
$$

Note  by  our choice of index $i$ for the tableaux $\T^-_{\ell,i}$  and $\T^+_{\ell,i}$
above,  any two edges that cross have different labels.   By replacing the crossings with rhombi as in \eqref{eq:rhombus}, which is equivalent to  \emph{trivalizing} the crossings
that was used earlier, we have the final flow diagram associated to $(\T^+, \T^-)$,  which is admissible.

$$
\grim{0.6}{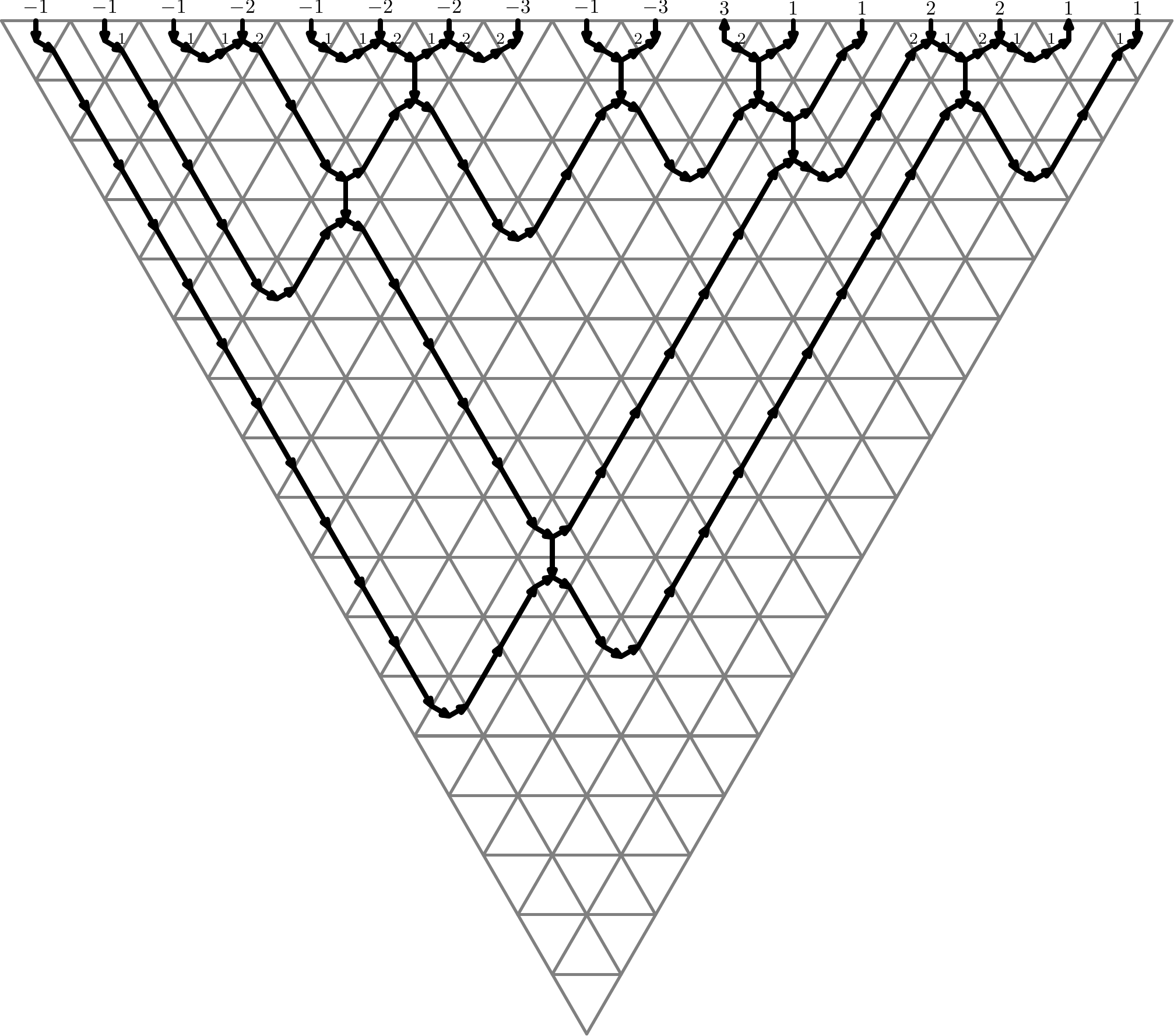}
$$

Reading the labels off the top  boundary, we can recover  $(\T^+, \T^-)$ from the flow diagram.
Hence, the map sending the pair  $(\T^+, \T^-)$  to the flow diagram is an injection, and indeed a bijection.


\begin{thebibliography}{BJS}
\bibitem[BJS]{BJS} S.~Billey, W.~ Jockusch, and R.~Stanley,
Some combinatorial properties of Schubert polynomials, \emph{J. Algebraic Combin.} \textbf{2} (1993), 345--374.

\bibitem[HRT]{HRT} M.~Housley, H.~Russell, and J.~Tymoczko, The Robinson-Schensted correspondence
and $A_2$-web bases,  arXiv: 1307.6487.

\bibitem[KK]{KK} M.~Khovanov  and G.~Kuperberg, Web bases for {${\rm sl}(3)$} are not dual canonical,
\emph{Pacific J. Math.}, \textbf{188} (1999), no. 1, 129--153.

\bibitem[K]{K} G.~Kuperberg, Spiders for rank 2 Lie algebras, \emph{Comm. Math. Phys.} \textbf{180} (1996), no. 1,  109--151.

\bibitem[PPR]{PPR} K.~Petersen, P.~Pylyavskyy, and B.~Rhoades, Promotion and cyclic sieving via webs, \emph{J. Algebraic Combin.} \textbf{30} (2009), no. 1, 19--41.

\bibitem[Ru]{Ru} H.~Russell,  An explicit bijection between semistandard tableaux and non-elliptic $sl_3$ webs, \emph{J. Algebraic Combin.} \textbf{38} (2013), no. 4,  851--862.

\bibitem[S]{S} B.~Sagan, {\it The Symmetric Group:   Representations, Combinatorial Algorithms, and Symmetric Functions},  Second edition. Graduate Texts in Mathematics, {\bf 203}. Springer-Verlag, New York, 2001.

\bibitem[St]{St} R.~Stanley, {\it Enumerative Combinatorics}, vol. 2, Cambridge University Press, Cambridge, 1999.

\bibitem[T]{T} J.~Tymoczko, A simple bijection between standard $3 \times n$ tableaux an irreducible webs for $\mathfrak{sl}_3$,  \emph{J. Algebraic Combin.} \textbf{35}  (2012), no. 4, 611--632.

\bibitem[We]{We} B.~Westbury, Web bases for the general linear groups, \emph{J. Algebraic Combin.} \textbf{35}  (2012), no. 1, 93--107.

\bibitem[W]{W} http://en.wikipedia.org/wiki/Robinson-Schensted$\_$correspondence

\bibitem[Wi]{Wi} H.~Wilf, The patterns of permutations, Kleitman and combinatorics: a celebration (Cambridge, MA, 1999), \emph{Discrete Math.} \textbf{257}  (2002) no.2-3, 575--583.
\end{thebibliography}
\end{document}